\documentclass[10pt,reqno,oneside]{amsart}

\usepackage[a4paper,margin=1.1in]{geometry}
\usepackage{amsmath,amssymb,amsthm,mathtools,mathrsfs}
\usepackage{thmtools}
\usepackage{enumitem}
\usepackage{hyperref}
\usepackage[nameinlink]{cleveref}
\usepackage{microtype}
\usepackage[dvipsnames]{xcolor}
\usepackage[normalem]{ulem}

\usepackage{comment}
\usepackage{cases}

\usepackage{tikz}
\usepackage{pgfplots}
\pgfplotsset{compat=1.18}
\usetikzlibrary{arrows.meta}

\definecolor{ValidColor}{HTML}{184E77}
\definecolor{InvalidColor}{HTML}{E6E8EB}
\definecolor{BoundaryColor}{HTML}{184E77}
\definecolor{RefColor}{HTML}{184E77}
\definecolor{AxisColor}{HTML}{24292F}
\definecolor{LegBorderColor}{HTML}{D7DBE0}
\tikzset{myarrow/.tip={Latex[scale=1.2, width=1.5mm]}}

\makeatletter
\renewcommand{\subsection}{\@startsection{subsection}{2}%
  \z@{\linespacing\@plus.7\linespacing}{.5\linespacing}%
  {\normalfont\scshape}}
\makeatother

\hypersetup{colorlinks=true,linkcolor=blue,citecolor=blue,urlcolor=blue}

\newcommand{\R}{\mathbb{R}}

\newcommand{\N}{\mathbb{N}}
\newcommand{\dd}{\,\mathrm{d}}

\newcommand{\supp}{\operatorname{supp}}
\newcommand{\loc}{\operatorname{loc}}
\newcommand{\eps}{\varepsilon}
\newcommand{\abs}[1]{\left|#1\right|}

\newcommand{\Lop}{\mathcal L}

\newcommand{\pc}{p_{*}}

\newtheorem{theorem}{Theorem}
\newtheorem{proposition}[theorem]{Proposition}
\newtheorem{lemma}[theorem]{Lemma}
\newtheorem{corollary}[theorem]{Corollary}
\newtheorem{remark}[theorem]{Remark}
\newtheorem{definition}[theorem]{Definition}

\title[]{Blow-up for a semilinear Tricomi equation in the oscillatory regime at the critical Strauss-type exponent}

\author{Diego Marcon, Wanderley Nunes do Nascimento, \and Matheus C. Santos}

\thanks{%
Instituto de Matemática e Estatística, Universidade Federal do Rio Grande do Sul, Porto Alegre, Brazil \\ 
}

\date{}

\begin{document}

\begin{abstract}
We study finite-time blow-up for a scale-invariant semilinear shifted Tricomi equation with decreasing time-dependent propagation speed and mass in the oscillatory regime.
We focus on the Strauss-type critical regime and prove that every weak solution with finite speed of propagation, arising from nonnegative nontrivial data satisfying a suitable localization condition, blows up in finite time. For sufficiently small initial data, we also obtain an exponential upper bound for the lifespan. More generally, our technique allows for a time-weighted power nonlinearity, a formulation that also includes, after a suitable transformation, Euler--Poisson--Darboux--Tricomi equations with scale-invariant damping and mass in the oscillatory regime. The proof relies on a positive self-similar solution of the homogeneous adjoint equation combined with a previously established lower bound for the nonlinear term and new estimates adapted to the critical case. This construction reduces the PDE problem to a nonlinear differential inequality and an ODE comparison argument yields both the finite-time blow-up and the lifespan estimate.

\bigskip


\end{abstract}

\maketitle

\section{Introduction}\label{sec:introduction}

In this paper we study the finite-time blow-up of solutions to the Cauchy
problem
\begin{equation}\label{eq:intro-main}
 \begin{cases}
  u_{tt}-(1+t)^{2\ell}\Delta u+\dfrac{\mu^{2}}{(1+t)^{2}}\,u=\abs{u}^{\pc},
  &(t,x)\in(0,T)\times\R^{n},\\
  u(0,x)=\eps f(x),\qquad u_{t}(0,x)=\eps g(x),
  &x\in\R^{n},
 \end{cases}
\end{equation}
where $\eps>0$, the initial data $f,g$ are nonnegative
and compactly supported, and the parameters of the equation satisfy
\[
 n\in\N,\qquad
 -1<\ell\leq0,\qquad
 \mu^{2}>\tfrac14,\qquad
 (1+\ell)n>1 .
\]
The power $\pc$ is taken to be the Strauss-type exponent
$\pc=\pc(n,\ell)>1$, defined as the unique positive root of the
quadratic polynomial
\begin{equation}\label{eq:intro-critical-polynomial}
\gamma(n,\ell;p)
:=
\Bigl((1+\ell)n-1\Bigr)p^{2}
-\Bigl((1+\ell)n-2\ell+1\Bigr)p
-2(1+\ell),
\end{equation}
The condition $-1<\ell\leq0$ implies that propagation speed
$a(t)=(1+t)^\ell$ is constant for $\ell=0$, and is decreasing and tends to zero for $\ell<0$. The term $\mu^2(1+t)^{-2}u$ is a scale-invariant time-dependent mass term, and the assumption $\mu^2>1/4$ corresponds to the oscillatory regime of the associated linear equation. Finally, $(1+\ell)n>1$ is the condition ensuring that the critical exponent defined above satisfies $\pc(n,\ell)>1$, and in particular, it implies $n\geq2$. The one-dimensional case $n=1$ is considered in \cite{MarconNascimentoSantos2026}, where the finite-time blow-up of solutions is established for every $p>1$. \

Problem \eqref{eq:intro-main} can be viewed as a particular instance of the
semilinear Cauchy problem
\begin{equation}\label{eq:intro-general}
\begin{cases}
 \Lop u=\abs{u}^{p},\\
 (u,u_t)\big|_{t=0}=\eps(f,g),
\end{cases}
\end{equation}
where $\Lop$ is a linear differential operator of second order with respect
to the time variable, whose coefficients may depend on $t$. We denote by
\[
 {T}_{\varepsilon}
 :=
 \sup\bigl\{
 T>0:\text{ a solution to \eqref{eq:intro-general} exists on }[0,T)
 \bigr\}
\]
the lifespan of the corresponding solution.

Over the last several decades, a substantial literature has been devoted to
problems of the form \eqref{eq:intro-general} for a wide variety of choices
of the operator $\Lop$. These include, among others, wave and Klein-Gordon
operators, equations with time-dependent propagation speed, damping or mass
terms, structurally damped and fractional evolution equations, as well as
models involving memory or other nonlocal effects. For broad classes of such equations, a recurring feature is the emergence of a critical-power phenomenon separating finite-time blow-up from small-data global existence. The location of this threshold is determined by the interplay between the long-time propagation and decay properties of the associated linear evolution, the influence of lower-order terms, and the class of initial data under consideration. Accordingly, in many models one can identify a critical exponent $p_{\mathrm c}>1$ whose value may depend on the operator $\Lop$ and the space dimension $n$, as well as on the functional framework and on assumptions concerning the initial data, such as their integrability, spatial decay, or compact support. Whenever such a sharp threshold is known, the following behavior is frequently observed:
\begin{itemize}
\item in the subcritical range $1<p<p_{\mathrm c}$, nontrivial solutions
corresponding to suitable, often nonnegative, initial data blow up in
finite time. In this regime, upper bounds for the lifespan are commonly of
polynomial type in $\eps^{-1}$;

\item at the critical exponent \(p=p_{\mathrm c}\), the behavior is model-dependent: global existence may hold, or solutions may blow up in finite time. When blow-up occurs, the transition from the subcritical regime is often reflected in a change in the lifespan scale, with upper bounds of exponential type in a negative power of \(\varepsilon\).

\item in the supercritical range $p>p_{\mathrm c}$, global-in-time solutions can
often be constructed for sufficiently small initial data in appropriate
function spaces.
\end{itemize}

Two classical examples for this phenomenon are the free and the dissipative wave equations. For
\[
 u_{tt}-\Delta u=\abs{u}^{p},
\]
the critical exponent is the Strauss exponent $p_{\mathrm S}(n)$, namely the positive root of the polynomial
\[ \gamma_{\mathrm S}(n;p):=(n-1)p^{2}-(n+1)p-2.\]
The corresponding blow-up and small initial data global-existence theory is developed in \cite{GeorgievLindbladSogge1997,Glassey1981a,John1979,  Schaeffer1985, Sideris1984,YordanovZhang2006,Zhou2007},
while sharp critical lifespan estimates are obtained in \cite{TakamuraWakasa2011, ZhouHan2014}. On the other hand, for the standard damped wave equation, 
\[ u_{tt}-\Delta u+u_t=\abs{u}^{p},\]
the dissipation is strong enough to make the large-time dynamics asymptotically parabolic. Therefore, the threshold occurs at the Fujita critical power $p_{\mathrm F}(n)=1+2/n$; see \cite{TodorovaYordanov2001, Zhang2001}. These two mechanisms provide the basic wave-like and heat-like reference regimes for semilinear equations with time-dependent coefficients.

For the present problem, a more closely related class is given by equations
with time-dependent propagation speed, such as the generalized Tricomi
equation (written in the time shifted form):
 \[u_{tt}-(1+t)^{m}\Delta u=\abs{u}^{p},
 \qquad m>0.\]
Its propagation geometry gives rise to the modified Strauss-type critical exponent $p_{\mathrm{ST}}(n,m)$, defined as the positive root of the polinomial
\[\gamma_{\mathrm{ST}}(n,m;p)
 :=
 \gamma_{\mathrm S}\left(\frac{(m+2)n}{2};p\right)
 +m(p-1)\]
or, formally, $\gamma_{\mathrm{ST}}(n,m;p) =
 \gamma_{\mathrm S}\left(
 \frac{(m+2)n}{2}+\frac{m}{p};p
 \right).$ The linear theory and the early global-existence results are developed by  \cite{Yagdjian2004, Yagdjian2006,  Yagdjian2007b,Yagdjian2007a}, while the corresponding blow-up/global-existence threshold is subsequently completed in the series of works
\cite{HeWittYin2016II,HeWittYin2017,HeWittYin2017a,
HeWittYin2018, HeWittYin2020}. Lifespan estimates in both the subcritical and the critical ranges are obtained in \cite{LinTu2019}.

The polynomial \eqref{eq:intro-critical-polynomial} is closely related
to the polynomials $\gamma_{\mathrm S}$ and
$\gamma_{\mathrm{ST}}$ through
\[
 \gamma(n,\ell;p) =\gamma_{\mathrm{ST}}(n,2\ell;p) = \gamma_{\mathrm S}\left((1+\ell)n+\frac{2\ell}{p};p\right),
\]
so that $\gamma(n,0;p)=\gamma_{\mathrm S}(n;p)$ and the classical Strauss relation is recovered when $\ell=0$. Moreover, the identification $m=2\ell$ yields that the polynomial $\gamma(n,\ell;p)$ provides an algebraic continuation of the Tricomi critical polynomial from the usual range $m>0$ to $-2<m\leq0$, corresponding to the decreasing propagation speed $a(t)=(1+t)^\ell$ considered in \eqref{eq:intro-main}.

The other essential feature of \eqref{eq:intro-main} is the
scale-invariant mass term. Together with the time-dependent propagation
speed, this connects the problem with the regular semilinear
Euler-Poisson-Darboux-Tricomi family
\begin{equation}\label{eq:intro-scale-invariant}
 u_{tt}-(1+t)^{m}\Delta u
 +\frac{\mu_{1}}{1+t}u_t
 +\frac{\mu_{2}^{2}}{(1+t)^{2}}u
 =\abs{u}^{p},
\end{equation}
for which the interaction between damping and mass is encoded by the
discriminant
\begin{equation}\label{eq:intro-delta}
 \delta:=(\mu_{1}-1)^{2}-4\mu_{2}^{2}.
\end{equation}
A substantial part of the available theory for the power nonlinearity has been developed under the assumption $\delta\geq0$. In this regime, one can construct time-dependent functions of fixed sign that solve the ordinary differential equation arising from the temporal part of the homogeneous adjoint problem. These functions can be incorporated into positive test functions or multipliers, a feature that is repeatedly exploited in blow-up arguments based on test function, comparison, and iteration methods. Depending on the relative influence of damping and mass, the threshold may then be governed by a heat-like or a wave-like mechanism. More precisely, the available theory reveals a competition between the shifted values
\begin{equation}\label{eq:shifted-F-T}
 p_{\mathrm{ST}}\!\left(
 n+\frac{2\mu_{1}}{m+2},m
 \right)
 \qquad\text{and}\qquad
    p_{\mathrm F}\!\left(
 \frac{(m+2)n+\mu_{1}-1-\sqrt{\delta}}{2}
 \right),
\end{equation} 
with their maximum arising as the natural blow-up threshold in several
parameter regimes. Blow-up, lifespan estimates and complementary global-existence results for different ranges of the parameters are obtained in \cite{NascimentoPalmieriReissig2017,PalmieriEven2019, PalmieriOdd2019, PalmieriReissig2018,
PalmieriReissig2019,PalmieriTu2019} for $m=0$. See also \cite{ DAbbicco2015,DLR2015} for the scale-invariant damping problem for $\mu_2=0$.  The genuinely variable-speed case has been studied more recently. Palmieri \cite{Palmieri2025} proves blow-up and lifespan estimates for $m>-2$ and $\delta\geq0$ below the larger of the shifted Fujita and Strauss-Tricomi exponents, also treating the Fujita-type borderline. The Tricomi critical case is addressed in \cite{LaiPalmieriTakamura2026}, while further critical lifespan estimates are obtained in \cite{LiGuo2026}. On the global-existence side, \cite{LiGuo2025} proves small-data global existence in a strongly dissipative regime, which, together with \cite{Palmieri2025}, identify the shifted Fujita exponent as critical in that parameter region.

When $\delta<0$, the usual blow-up arguments face a loss of the positivity mechanism: the standard temporal profiles associated with the homogeneous adjoint problem become oscillatory, so that positive test functions and multipliers can no longer be constructed in the same direct way as for $\delta\geq0$. Overcoming this obstruction therefore requires more specialized adjoint constructions, and the nonlinear theory in this regime is consequently much less developed.

For $m=0$, \cite{DAbbiccoPalmieri2021} proves small-data global existence for $p$ above an explicitly determined threshold. Blow-up results for derivative-type nonlinearities have also been obtained more recently. For the nonlinearity $\abs{u_t}^{p}$, finite-time blow-up and lifespan estimates are established in \cite{Hamza2026},  while \cite{ChenHamza2026} extends this analysis to more general time-dependent damping and mass coefficients, including the critical case and without imposing a sign condition on $\delta$.

The relation between the regime $\delta<0$ and the equation considered in
the present paper {becomes more transparent after removing the
scale-invariant damping term from the full semilinear equation
\eqref{eq:intro-scale-invariant}. Indeed, applying the transformation}
\[
 u(t,x)=(1+t)^{-\mu_{1}/2}w(t,x)
\] {gives}
\begin{equation}\label{eq:intro-transformed-weighted}
 w_{tt}-(1+t)^{m}\Delta w
 +\frac{\mu^{2}}{(1+t)^{2}}w
 = {
 (1+t)^{-\alpha}\abs{w}^{p},}
\end{equation} {with}
\[
 \mu^{2}
 =
 \mu_{2}^{2}
 +\frac{\mu_{1}}{2}
 -\frac{\mu_{1}^{2}}{4}
 \qquad 
{\text{and} \qquad
 \alpha
 =
 \frac{\mu_{1}}{2}(p-1).}
\] The transformed mass parameter satisfies
$4\mu^{2}-1=-\delta$. 
Therefore, identifying $m=2\ell$, {the assumption
$\delta<0$ is equivalent to $\mu^{2}>1/4$, and the Euler--Poisson--Darboux--Tricomi
equation is transformed into an equation with the same linear
operator as \eqref{eq:intro-main}, but with the time-weighted power
nonlinearity $(1+t)^{-\alpha}|w|^p$. In particular, when $\mu_{1}=0$, we have $\alpha=0$ and \eqref{eq:intro-transformed-weighted} reduces to the
unweighted model \eqref{eq:intro-main}.}

{This naturally leads us to consider the more general weighted
critical polynomial}
\begin{equation}\label{eq:intro-weighted-critical-polynomial}
{\gamma_\alpha(n,\ell;p)
:=
\Bigl((1+\ell)n-1\Bigr)p^2
-
\Bigl((1+\ell)n+1-2\ell-2\alpha\Bigr)p
-
2(1+\ell).}
\end{equation}
{We notice that
$\gamma_0(n,\ell;p)=\gamma(n,\ell;p)$.
Formally, for the value $\alpha=\mu_1(p-1)/2$ we obtain}
\[
{\begin{aligned}
\gamma_{\frac{\mu_1}{2}(p-1)}(n,\ell;p)
&=
\Bigl(n(1+\ell)+\mu_1-1\Bigr)p^2 
-\Bigl(n(1+\ell)+\mu_1+1-2\ell\Bigr)p
-2(1+\ell)\\
&=
\gamma_{\mathrm{ST}}\left(
n+\frac{\mu_1}{1+\ell}, 2 \ell;p
\right).
\end{aligned}}
\] {Thus, under the change of variables above, the weighted critical polynomial reduces to the shifted Strauss--Tricomi polynomial associated with \eqref{eq:intro-scale-invariant}, see \eqref{eq:shifted-F-T}.}


The blow-up theory in this regime has been recently developed in \cite{MarconNascimentoSantos2026}. For \eqref{eq:intro-scale-invariant}, with decreasing propagation speed $-2<m<0$ and $\delta<0$, the authors prove finite-time blow-up below the shifted Strauss-Tricomi exponent
\[
 p_{\mathrm{ST}}\!\left(
 n+\frac{2\mu_{1}}{m+2},m
 \right),
\]
whenever this exponent is larger than one. The corresponding result for
the undamped problem \eqref{eq:intro-main}, obtained by setting
$\mu_{1}=0$ and $m=2\ell$, gives blow-up throughout
\[
 1<p<\pc(n,\ell).
\] Nevertheless, the borderline value $p=\pc(n,\ell)$ is not covered by \cite{MarconNascimentoSantos2026}.

{The aim of the present paper is first to settle this borderline
case. We prove, see \Cref{thm:main}, finite-time blow-up for the equation with nonlinearity
$(1+t)^{-\alpha}|u|^{\pc}$ and establish the critical lifespan estimate}
\begin{equation}\label{eq:intro-lifespan}
 T_{\eps}
 \leq
 \exp\!\left(
 C\eps^{-\pc(\pc-1)}
 \right).
\end{equation}
This completes the blow-up side of the expected Strauss-type
threshold up to and including the borderline exponent, but it does not
yet provide a complete identification of the critical exponent. Indeed,
the complementary small-data global-existence problem for
$p>\pc(n,\ell)$ remains open when $\mu^{2}>1/4$.

The current state of problem \eqref{eq:intro-main} is summarized in
\Cref{fig:blowup-region}. For fixed $n$, the dashed line
$n(1+\ell)=1$ separates two qualitatively different regimes. When
$n(1+\ell)\leq1$, finite-time blow-up is known for every $p>1$ by
\cite{MarconNascimentoSantos2026}. When $n(1+\ell)>1$, the equation
$\gamma(n,\ell;p)=0$ defines the Strauss-type exponent
$\pc(n,\ell)>1$. In this regime, the blue region
$\gamma(n,\ell;p)<0$ corresponds to the previously known 
blow-up range, whereas the solid curve $\gamma(n,\ell;p)=0$ is the
Strauss-critical case established in the present paper. The problem remains open in the gray region $\gamma(n,\ell;p)>0$.

\begin{figure}[htbp]
\centering
\begin{tikzpicture}
    \pgfmathsetmacro{\myN}{10}
    \pgfmathsetmacro{\pmax}{7}
    \pgfmathsetmacro{\xmax}{\pmax + 0.25}

    \pgfmathsetmacro{\asy}{-(\myN - 1) / \myN}

    \pgfmathsetmacro{\proot}{
        ( (\myN + 1) + sqrt( (\myN + 1)^2 + 8*(\myN - 1) ) )
        / ( 2*(\myN - 1) )
    }

    \begin{axis}[
        width=10cm,
        height=7.5cm,
        xmin=1, xmax=\xmax,
        ymin=-1, ymax=0.05,
        axis x line=bottom,
        axis y line=left,
        axis line style={
            AxisColor,
            line width=1.2pt,
            line join=round
        },
        x axis line style={-{myarrow}},
        y axis line style={-},
        major tick length=4.5pt,
        xtick={1},
        xticklabels={1},
        xtick style={AxisColor, line width=1.2pt},
        tick label style={font=\normalsize, color=AxisColor},
        ytick={-1, \asy, 0},
        yticklabels={$-1$, $-1+\frac{1}{n}$, $0$},
        ytick style={AxisColor, line width=1.2pt},
        xlabel={$p$},
        ylabel={$\ell$},
        xlabel style={font=\Large, color=AxisColor},
        ylabel style={font=\Large, color=AxisColor, rotate=-90},
        enlargelimits=false,
        clip=true,
        declare function={
            crit(\x) =
            ((\myN + 1)*\x + 2 - (\myN - 1)*\x^2)
            / ((\x - 1)*(\myN*\x + 2));
        },
        legend style={
            at={(0.9, 0.9)},
            anchor=north east,
            font=\small,
            draw=LegBorderColor,
            line width=0.75pt,
            fill=white,
            fill opacity=0.94,
            text opacity=1,
            cells={anchor=west},
            rounded corners=1.5pt
        }
    ]

    \fill[ValidColor, opacity=0.5]
        (1, -1) -- (1, 0) -- (\proot, 0)
        -- plot[domain=\proot:\pmax, samples=200]
            (\x, {crit(\x)})
        -- (\pmax, -1) -- cycle;

    \fill[InvalidColor, opacity=0.95]
        (\proot, 0)
        -- plot[domain=\proot:\pmax, samples=200]
            (\x, {crit(\x)})
        -- (\pmax, 0) -- cycle;

    \draw[
        RefColor,
        opacity=0.6,
        line width=1.25pt,
        dash pattern=on 4.0pt off 3.0pt
    ]
        (1, \asy) -- (\pmax, \asy);

    \draw[
        BoundaryColor,
        line width=2pt,
        line cap=round,
        line join=round
    ]
        plot[domain=\proot:\pmax, samples=200]
            (\x, {crit(\x)});

    \addlegendimage{
        area legend,
        fill=ValidColor,
        fill opacity=0.28,
        draw=none
    }
    \addlegendentry{$\gamma(n,\ell;p)<0$}

    \addlegendimage{
        area legend,
        fill=InvalidColor,
        fill opacity=0.95,
        draw=none
    }
    \addlegendentry{$\gamma(n,\ell;p)>0$}

    \addlegendimage{
        line legend,
        draw=BoundaryColor,
        line width=2.35pt
    }
    \addlegendentry{$\gamma(n,\ell;p)=0$}

    \addlegendimage{
        line legend,
        draw=RefColor,
        line width=1.25pt,
        dash pattern=on 4.0pt off 3.0pt
    }
    \addlegendentry{$n(\ell+1)=1$}

    \end{axis}
\end{tikzpicture}
\caption{
Known blow-up region in the $(p,\ell)$-plane for fixed $n=10$,
with $p\geq1$ and $-1<\ell\leq0$.
}
\label{fig:blowup-region}
\end{figure}
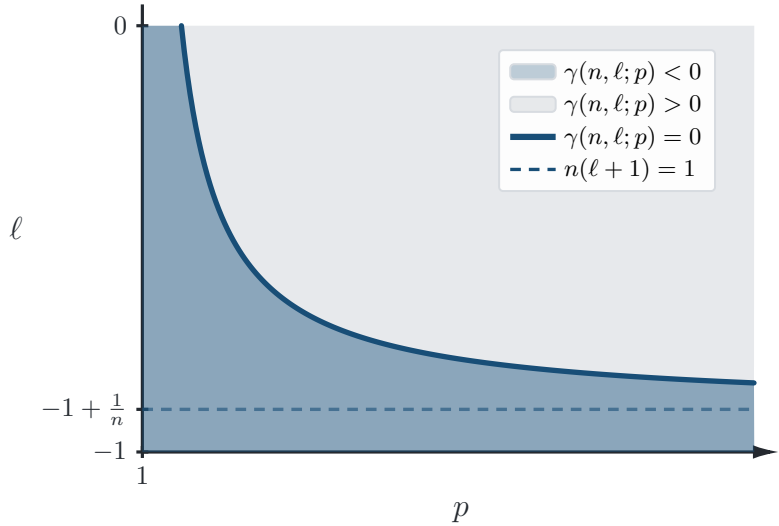

The main strategy of the proof is to reduce the PDE problem to a critical
ordinary differential inequality to which an ODE comparison principle can
be applied.  This type of comparison argument is introduced by Zhou and
Han \cite{ZhouHan2014} in the study of critical semilinear wave equations
and has subsequently been used in related critical blow-up and lifespan
problems; see, for instance, \cite{IkedaSobajima2018,PalmieriTu2019}.
To carry out this reduction, we introduce a nonlinear functional
with a weight which is a positive self-similar solution of the homogeneous adjoint
equation expressed in terms of the Gauss hypergeometric function.  Such hypergeometric adjoint and test-function constructions have proved useful
in critical blow-up problems, see \cite{IkedaSobajima2018,
IkedaSobajimaWakasa2019,LiGuo2026,PalmieriTu2019,Zhou1992,Zhou2007,ZhouHan2014}. Then, we combine this adjoint construction with  lower estimates previously obtained in \cite{MarconNascimentoSantos2026}.  Together with new estimates in the critical case, these bounds allow us to derive the aforementioned differential inequality.

The paper is organized as follows.  In \Cref{sec:preliminaries}, we
introduce the notation and the solution class and state the main result.
In \Cref{sec:general-adj-identity}, we establish a general moment identity for solutions of the homogeneous adjoint equation, which provides the basic link between the PDE and the weighted functionals used later.  In
\Cref{sec:hypergeom}, we construct the positive self-similar adjoint
functions in terms of the Gauss hypergeometric function and establish the
estimates needed in the critical argument.  In
\Cref{sec:critical-inequality}, we combine these adjoint estimates with the
lower bound from \cite{MarconNascimentoSantos2026}, introduce the relevant
weighted and time-integrated functionals, and derive the critical
differential inequality.  In \Cref{sec:ode}, we prove the ODE comparison
lemma used to control the lifespan.  Finally, in \Cref{sec:proof-main}, we
apply this comparison principle to complete the proof of the main theorem
and obtain the exponential upper bound for the lifespan.

\section{Notations and main result}\label{sec:preliminaries}

We begin by introducing the notation and the main auxiliary objects employed in
the paper. In particular, we recall the polynomial defining the
borderline exponent, specify the coefficients and propagation quantities
associated with the linear operator, and introduce the spatial weights used
in the construction of the separated adjoint solutions.

Throughout the paper, we use the following definitions and conventions:
\begin{enumerate}[label=$(\mathrm{D}\arabic*)$]

\item\label{def:critical-exponent}
\textup{(Critical polynomial and critical exponent)}
Recall from \eqref{eq:intro-weighted-critical-polynomial} that
\[
     {\gamma_\alpha(n,\ell;p)}
    =
    \Bigl((1+\ell)n-1\Bigr)p^2
    - 
     {\Bigl((1+\ell)n+1-2\ell-2\alpha\Bigr)p}
    -
    2(1+\ell).
\]
 {Throughout the weighted critical argument, $\alpha\geq0$ is fixed and $\pc>1$ is assumed to satisfy $\gamma_\alpha(n,\ell;\pc)=0$.}  {When $\alpha=0$ and $(1+\ell)n>1$, this reduces to the unique positive root $\pc(n,\ell)$ of $\gamma(n,\ell;p)=\gamma_0(n,\ell;p)=0$, given explicitly by}
\begin{equation}\label{eq:critical-root}
    \pc(n,\ell)
    =
    \frac{
        (1+\ell)n+1-2\ell
        +
        \sqrt{
            \bigl((1+\ell)n+1-2\ell\bigr)^2
            +
            8(1+\ell)\bigl((1+\ell)n-1\bigr)
        }
    }{
        2\bigl((1+\ell)n-1\bigr)
    }.
\end{equation}

\item\label{def:operator-coefficients}
\textup{(Velocity, propagation, and potential)}
We set
\begin{equation}\label{eq:a-A-V}
    a(t):=(1+t)^\ell,
    \qquad
    A(t):=\int_0^t a(s)\dd s
    =
    \frac{(1+t)^{1+\ell}-1}{1+\ell},
    \qquad
    V(t):=\frac{\mu^2}{(1+t)^2}.
\end{equation}
Thus, the linear operator in \eqref{eq:intro-main} can be written as
\[
    \Lop
    :=
    \partial_t^2-a(t)^2\Delta+V(t).
\]
Since $-1<\ell\leq0$, the velocity $a$ is positive and nonincreasing,
whereas $A$ is strictly increasing and satisfies $A(t)\to+\infty$ as $t\to+\infty$.

\item\label{def:propagation-radius}
\textup{(Propagation quantities)}
For a fixed radius $R>0$, we define
\[
    \rho(t):=R+A(t).
\]
In the solution class introduced below, $R$ is an upper bound for the
support radius of the initial data, and $\rho(t)$ describes the
corresponding propagation radius at time $t$.

Whenever
\[
    0<R<\frac{1}{1+\ell},
\]
we also set
\[
    d_R:=1-(1+\ell)R>0.
\]
As shown in \eqref{eq:cone-gap}, this quantity measures the separation
between the propagation region and the boundary of the self-similar cone.

\item\label{def:oscillatory-parameter}
\textup{(Oscillatory parameter)}
Under the assumption $\mu^2>1/4$, we define
\[
    \omega:=\sqrt{\mu^2-\frac14}>0.
\]
This parameter appears in the order of the modified Bessel functions and as parameters of the hypergeometric adjoint
profiles.

\item\label{def:spatial-notation}
\textup{(Spatial notation and conjugate exponents)}
For $r>0$, we write
\[
    B_r:=\{x\in\R^n:|x|<r\}.
\]
For every $p>1$, we denote its Hölder conjugate by $p':=\frac{p}{p-1}$.


\item\label{def:constants}
The symbols $C$ and $c$ denote generic positive constants whose values may
change from line to line. Unless otherwise stated, these constants may depend
on the fixed parameters of the problem and on the initial profiles $f$ and
$g$, but are independent of $\eps$ and of the lifespan. Any additional
dependence is indicated whenever relevant.

\end{enumerate}


\begin{definition}[Weak solution with finite speed of propagation]\label{def:solution}
Let $T\in(0,\infty]$, $p>1$, $\eps>0$, $f\in H^1(\R^n)$, and $g\in L^2(\R^n)$.  A function $u$ is a weak solution with finite speed of propagation on $[0,T)$ for 
\begin{equation}\label{eq:main-p}
 \begin{cases}
 u_{tt}-a(t)^2\Delta u+V(t)u=  {(1+t)^{-\alpha}\abs{u}^p},
     &(t,x)\in\R^n\times(0,T),\\[0.5em]
 u(0,x)=\eps f(x),\quad u_t(0,x)=\eps g(x),
     &x\in\R^n,
 \end{cases}
\end{equation}
if the following conditions hold:
\begin{enumerate}[label=$(\roman*)$]
\item $u\in C^1([0,T);L^1_{\loc}(\R^n))$ and $|u|^p\in L^1_{\loc}([0,T)\times\R^n)$;
\item there exists $R>0$ such that
$\supp f\cup \supp g\subset B_{R}$,
and, for all $t\in[0,T)$,
\begin{equation}\label{eq:finite-propagation}
\supp u(t,\cdot)\subset B_{R+A(t)};
\end{equation}
\item for every $\psi\in C_c^\infty([0,T)\times\R^n)$,
\begin{align}\label{eq:weak-form}
 &\int_0^T\!\int_{\R^n}
 u(t,x)\Bigl(\psi_{tt}(t,x)-a(t)^2\Delta\psi(t,x)+V(t)\psi(t,x)\Bigr)\dd x\dd t
 \notag\\
 &\qquad=
 \int_0^T\!\int_{\R^n}  {(1+t)^{-\alpha}\abs{u(t,x)}^p}\psi(t,x)\dd x\dd t
 +\eps\int_{\R^n}\Bigl(g(x)\psi(0,x)-f(x)\psi_t(0,x)\Bigr)\dd x.
\end{align}
\end{enumerate}
\end{definition}

\begin{remark}\label{rem:ut-support}
The regularity in \Cref{def:solution} and the continuity of
$t\mapsto R+A(t)$ imply
\[
    \supp u_t(t,\cdot)
    \subset
    \overline{B_{R+A(t)}}.
\]
Indeed, every compact set contained in the complement of this closed ball
remains outside the propagation ball for all times in a neighborhood of
$t$. On such a compact set, the restriction of $u$ is therefore identically
zero as an $L^1$-valued function, and hence so is its time derivative.
\end{remark}

Having fixed the notion of solution, we can now state our main result.

\begin{theorem}[Critical blow-up and lifespan estimate]
\label{thm:main}
Let $n\in\N$ and $-1<\ell\leq0$. Let $\mu^2>1/4$ and  {let $\alpha\geq0$ and $\pc>1$ satisfy $\gamma_\alpha(n,\ell;\pc)=0$}. Suppose that $f\in H^1(\R^n)$ and $g\in L^2(\R^n)$ satisfy
\begin{equation}\label{eq:data-assumptions}
    f,g\geq0\ \text{a.e.},
    \qquad
    (f,g)\neq(0,0),
     {\qquad
    g+\left(\frac{(1+\ell)n-1}{2}-\frac{1+\ell}{\pc}\right)f\geq0\ \text{a.e.},}
\end{equation}
and that, for some $0<R<(1+\ell)^{-1}$, the initial data satisfy $\supp f\cup\supp g    \subset \overline{B_R}$.

Then, for every $\eps>0$, any weak solution of \eqref{eq:intro-main} with finite speed of propagation, in the sense of \Cref{def:solution}, blows up in finite
time, that is, its maximal existence time $T_\eps$ is finite.

Moreover, there exist constants $\eps_0>0$ and $C>0$, independent of
$\eps$, such that, for every $0<\eps\leq\eps_0$,
\begin{equation}\label{eq:lifespan-exp}
    T_\eps
    \leq
    \exp \bigl(C\eps^{-\pc(\pc-1)}\bigr).
\end{equation}
\end{theorem}

\begin{remark}
    The last condition in \eqref{eq:data-assumptions} is automatically satisfied when $\alpha=0$; see \Cref{rem:data-condition}.
\end{remark}

 {We conclude this section with an application to the scale-invariant damped problem. The weighted formulation above allows us to complement Corollary~7 of \cite{MarconNascimentoSantos2026} by settling the previously open Strauss-type critical case in the oscillatory regime.}
\begin{corollary}[{Critical case for the damped problem}]\label{cor:damped-critical}
{Let $n\in\N$, $-1<\ell\leq0$, $\mu_1\geq0$, and $\mu_2\in\R$. Assume that $\delta:=(1-\mu_1)^2-4\mu_2^2<0$ and that $n(1+\ell)+\mu_1>1$. Let $p>1$ be the unique root of}
\begin{equation}\label{eq:damped-critical-polynomial}
 {\Bigl(n(1+\ell)+\mu_1-1\Bigr)p^2-\Bigl(n(1+\ell)+\mu_1+1-2\ell\Bigr)p-2(1+\ell)=0.}
\end{equation}
 {Assume further that $f\in H^1(\R^n)$, $g\in L^2(\R^n)$, $f,g\geq0$ a.e., $(f,g)\neq(0,0)$, and $\supp f\cup\supp g\subset\overline{B_R}$ for some $0<R<(1+\ell)^{-1}$. Then, every weak solution with finite speed of propagation of}
\begin{equation}\label{eq:damped-critical-problem}
 {\begin{cases}
w_{tt}-(1+t)^{2\ell}\Delta w+\dfrac{\mu_1}{1+t}w_t+\dfrac{\mu_2^2}{(1+t)^2}w=|w|^p,
&(t,x)\in(0,T)\times\R^n,\\[1em]
w(0,x)=\eps f(x),\qquad w_t(0,x)=\eps g(x),
&x\in\R^n,
\end{cases}}
\end{equation}
 {blows up in finite time. Moreover, there exist $\eps_0,C>0$, independent of $\eps$, such that, for every $0<\eps\leq\eps_0$,}
\[
 {T_\eps\leq\exp\!\bigl(C\eps^{-p(p-1)}\bigr).}
\]
\end{corollary}

\begin{proof}[ {Proof}]
 {Set $u(t,x):=(1+t)^{\mu_1/2}w(t,x)$. Then, $u$ solves \eqref{eq:main-p} with}
\[
 {\alpha=\frac{\mu_1}{2}(p-1),\qquad \widetilde\mu^2:=\mu_2^2+\frac{\mu_1(2-\mu_1)}4=\frac{1-\delta}{4}>\frac14,}
\]
 {and initial data}
\[
 {u(0,x)=\eps f(x),\qquad u_t(0,x)=\eps\left(g(x)+\frac{\mu_1}{2}f(x)\right).}
\]
 {For this value of $\alpha$, the relation $\gamma_\alpha(n,\ell;p)=0$ is exactly \eqref{eq:damped-critical-polynomial}. It remains only to verify \eqref{eq:data-assumptions}. Set}
\[
 {\widehat\beta_p:=\frac{n(1+\ell)+\mu_1+1}{2}-\frac{1+\ell}{p}.}
\]
 {Writing $d:=n(1+\ell)+\mu_1$, equation \eqref{eq:damped-critical-polynomial} yields}
\[
 {2p(\widehat\beta_p-1)=p(d-1)-2(1+\ell)=d-1-4\ell+\frac{2(1+\ell)}p>0.}
\]
{Hence, $\widehat\beta_p>1$ and, therefore,}
\[
 {g+\frac{\mu_1}{2}f+\left(\frac{(1+\ell)n-1}{2}-\frac{1+\ell}{p}\right)f=g+(\widehat\beta_p-1)f\geq0.}
\]
{Thus, all assumptions of \Cref{thm:main} are satisfied by $u$. The conclusion follows, and the transformation does not change the lifespan.}
\end{proof}


\section{A general adjoint moment identity}\label{sec:general-adj-identity}

In this section, we establish a general moment identity for solutions of the
homogeneous adjoint equation. This identity provides the link between the
semilinear problem and the weighted functionals used in the critical
blow-up argument. We later apply it to the positive self-similar
adjoint function constructed in \Cref{sec:hypergeom}.

Recall that the linear operator $\Lop =\partial_t^2-a(t)^2\Delta+V(t)$ is formally self-adjoint. We thus consider a real-valued function $\Phi$ satisfying
\[
    \Lop\Phi=0.
\]
For $\tau<T$, define the propagation set
\begin{equation}\label{def:Ktau}
    K_\tau
    :=
    \left\{
        (t,x)\in[0,\tau]\times\R^n:
        |x|\leq R+A(t)
    \right\}.
\end{equation}
By finite speed of propagation, all pairings involving $u$, $u_t$, or
 {$(1+t)^{-\alpha}|u|^p$} on $[0,\tau]$ are supported in $K_\tau$. Consequently, it is
sufficient to assume that $\Phi$ is smooth and satisfies $\Lop\Phi=0$ in a
neighborhood of this set.

Associated with $\Phi$, we define the weighted moment
\begin{equation}\label{eq:M-Phi}
    M_\Phi(t)
    :=
    \int_{\R^n}u(t,x)\Phi(t,x)\dd x,
\end{equation}
the Wronskian moment
\begin{equation}\label{eq:W-Phi}
    W_\Phi(t)
    :=
    \int_{\R^n}
    \bigl(
        u_t(t,x)\Phi(t,x)
        -
        u(t,x)\Phi_t(t,x)
    \bigr)\dd x,
\end{equation}
and the nonlinear weighted moment
\begin{equation}\label{eq:G-Phi}
    G_\Phi(t)
    :=
     {(1+t)^{-\alpha}}\int_{\R^n}|u(t,x)|^p\Phi(t,x)\dd x.
\end{equation}

To motivate the identity, suppose formally that $u$ and $\Phi$ are
sufficiently smooth. Differentiating the Wronskian moment and using
  {$\Lop u=(1+t)^{-\alpha}|u|^p$} and $\Lop\Phi=0$, we obtain
\begin{align*}
    W_\Phi'(t)
    =
    \int_{\R^n}
    \bigl(u_{tt}\Phi-u\Phi_{tt}\bigr)\dd x
    =
    a(t)^2
    \int_{\R^n}
    \bigl(\Delta u\,\Phi-u\Delta\Phi\bigr)\dd x
    + 
     {(1+t)^{-\alpha}\int_{\R^n}|u|^p\Phi\dd x}
    =
    G_\Phi(t),
\end{align*}
where the potential terms cancel and the last equality follows from spatial
integration by parts. In particular, if $\Phi\geq0$ on the propagation
region, then $G_\Phi\geq0$, and hence $W_\Phi$ is nondecreasing.

The following lemma justifies this computation directly from the weak
formulation and records both its differential and integrated consequences.
These identities are later applied in \Cref{sec:critical-inequality}.

\begin{lemma}\label{lem:moment-identity}
Let $u$ be a weak solution to \eqref{eq:main-p} on  $[0,T)$ in the sense of \Cref{def:solution}. Assume that, for every $\tau<T$, the function $\Phi$ is real-valued and smooth in a neighborhood of $K_\tau$ defined in \eqref{def:Ktau} and that it satisfies
\[ \Lop\Phi=0.\]
Let $M_\Phi$, $W_\Phi$, and $G_\Phi$ be the moments defined in \eqref{eq:M-Phi}, \eqref{eq:W-Phi}, and \eqref{eq:G-Phi}, respectively.
Then, we have $M_\Phi\in C^1([0,T))$, $W_\Phi\in W^{1,1}_{\loc}([0,T))$, and
\begin{equation}\label{eq:M-prime-general}
  M_\Phi'(t)
    =
    W_\Phi(t)
    +
    2\int_{\R^n}u(t,x)\Phi_t(t,x)\dd x,
\end{equation}
for every $t\in[0,T)$, while
\begin{equation}\label{eq:W-prime-general}
    W_\Phi'(t)=G_\Phi(t)
\end{equation}
for almost every $t\in(0,T)$. Consequently,
\begin{equation}\label{eq:W-integrated-general}
    W_\Phi(t)=W_\Phi(0)+\int_0^tG_\Phi(s)\dd s
\end{equation}
and
\begin{equation}\label{eq:moment-integrated-general}
    M_\Phi(t)
    -
    2\int_0^t\int_{\R^n}
    u(s,x)\Phi_s(s,x)\dd x\dd s
    =
    M_\Phi(0)
    +
    tW_\Phi(0)
    +
    \int_0^t(t-s)G_\Phi(s)\dd s.
\end{equation}
For future reference, we also record that the initial moments are
\begin{equation}\label{eq:initial-moments-general}
    M_\Phi(0)
    =
    \eps\int_{\R^n}f(x)\Phi(0,x)\dd x
\end{equation}
and
\begin{equation}\label{eq:initial-W-general}
    W_\Phi(0)
    =
    \eps\int_{\R^n}
    \bigl(
        g(x)\Phi(0,x)
        -
        f(x)\Phi_t(0,x)
    \bigr)\dd x.
\end{equation}
In particular, if $\Phi\ge0$ on the propagation region, we have that
$W_\Phi$ is nondecreasing.
\end{lemma}

\begin{proof}
Fix $0<\tau<T$. Choose a function
$\chi\in C_c^\infty([0,T)\times\R^n)$ which is equal to one in a
neighborhood of $K_\tau$ and whose support is contained in a region where
$\Phi$ is defined. By finite propagation, in every pairing with $u$,
$u_t$, or  {$(1+t)^{-\alpha}|u|^p$} over $[0,\tau]$, the function $\chi\Phi$ may be replaced
by $\Phi$.

Let $\eta\in C_c^\infty(0,\tau)$ and use
\[
    \psi(t,x)=\eta(t)\chi(t,x)\Phi(t,x)
\]
as a test function in \eqref{eq:weak-form}. Since
$\Lop\Phi=0$, on the support of $u$ we have
\[
    \Lop(\eta\Phi)
    =
    \eta''\Phi+2\eta'\Phi_t.
\]
It follows that
\begin{equation}\label{eq:distribution-moment}
    \int_0^\tau M_\Phi(t)\eta''(t)\dd t
    +
    2\int_0^\tau
    \left(
        \int_{\R^n}u(t,x)\Phi_t(t,x)\dd x
    \right)\eta'(t)\dd t
    =
    \int_0^\tau G_\Phi(t)\eta(t)\dd t.
\end{equation}
Using the cutoff above, and since 
$u \in C^1([0,\tau];L^1_{\loc}(\R^n))$ while $\Phi$ is smooth in a neighborhood of $K_\tau$, the usual product rule for the corresponding $L^1$-pairing gives
\[
    M_\Phi'(t)
    =
    \int_{\R^n}
    \bigl(u_t\Phi+u\Phi_t\bigr)\dd x,
\]
which is precisely \eqref{eq:M-prime-general}. Hence, in the sense of
distributions,
\begin{align*}
    \langle W_\Phi',\eta\rangle
    &=
    -\int_0^\tau W_\Phi(t)\eta'(t)\dd t\\
    &=
    \int_0^\tau M_\Phi(t)\eta''(t)\dd t
    +
    2\int_0^\tau
    \left(
        \int_{\R^n}u(t,x)\Phi_t(t,x)\dd x
    \right)\eta'(t)\dd t\\
    &=
    \int_0^\tau G_\Phi(t)\eta(t)\dd t.
\end{align*}
Therefore, $W_\Phi'=G_\Phi$ in the sense of distributions. Since
$G_\Phi\in L^1(0,\tau)$, we obtain
$W_\Phi\in W^{1,1}(0,\tau)$ and \eqref{eq:W-integrated-general} follows.

Finally, integrating \eqref{eq:M-prime-general} and then using
\eqref{eq:W-integrated-general}, we find
\begin{equation*}
    M_\Phi(t)-M_\Phi(0)
    =
    tW_\Phi(0)
    +
    \int_0^t(t-s)G_\Phi(s)\dd s+
    2\int_0^t\int_{\R^n}
    u(s,x)\Phi_s(s,x)\dd x\dd s,
\end{equation*}
which is \eqref{eq:moment-integrated-general}. The formulas for the initial
moments follow directly from the initial conditions. If $\Phi\ge0$, then
$G_\Phi\ge0$, so \eqref{eq:W-prime-general} shows that $W_\Phi$ is
nondecreasing.
\end{proof}

\section{Positive self-similar adjoint functions}\label{sec:hypergeom}

In this section, we construct a family of positive self-similar solutions
of the homogeneous adjoint equation in terms of the Gauss hypergeometric
function. These functions are adapted to the geometry of the propagation
region and provide the weights used in the critical blow-up argument.
We then identify the profile corresponding to the critical exponent and
establish the properties of this profile, including the interior and borderline estimates.

Let us introduce the normalized radial variable
\begin{equation}\label{defofy}
    y=y(t,x)
    :=
    \frac{(1+\ell)\abs{x}}{(1+t)^{1+\ell}}.
\end{equation}
The level set $y=1$ determines the boundary of the open self-similar cone
\begin{equation}\label{eq:selfsimilar-cone}
    \mathcal C_\ell
    :=
    \big\{
        (t,x)\in(-1,\infty)\times\R^n:
        y(t,x)<1
    \big\},
\end{equation}
whose relevant portion is
\[
    \mathcal C_\ell^+
    :=
    \mathcal C_\ell
    \cap
    \bigl([0,\infty)\times\R^n\bigr).
\]
Thus, $0\leq y<1$ in the interior of $\mathcal C_\ell$.

If $u$ is a weak solution to \eqref{eq:main-p} in the sense of Definition \ref{def:solution}, then the propagation region of $u$ lies strictly inside $\mathcal C_\ell^+$. Indeed, setting $d_R:=1-(1+\ell)R>0$ as in \ref{def:propagation-radius},
the finite propagation property \eqref{eq:finite-propagation} yields, for
$x\in\supp u(t,\cdot)$,
\begin{equation}\label{eq:cone-gap}
\begin{aligned}
    y(t,x)
    =
    \frac{(1+\ell)\abs{x}}{(1+t)^{1+\ell}}
    \leq
    \frac{(1+\ell)R+(1+t)^{1+\ell}-1}
         {(1+t)^{1+\ell}}
    =
    1-\frac{d_R}{(1+t)^{1+\ell}}.
\end{aligned}
\end{equation}
Equivalently, the spatial gap between the self-similar cone and the
propagation region is constant:
\[
    \frac{(1+t)^{1+\ell}}{1+\ell}
    -
    \bigl(R+A(t)\bigr)
    =
    \frac{1}{1+\ell}-R>0.
\]

For $\beta\in\R$, define
\begin{equation}\label{eq:hyp-parameters}
    a_\beta
    :=
    \frac{2\beta-1}{4(1+\ell)}
    +
    \frac{i\omega}{2(1+\ell)},
    \qquad
    b_\beta
    :=
    \frac{2\beta-1}{4(1+\ell)}
    -
    \frac{i\omega}{2(1+\ell)},
    \qquad
    c:=\frac n2,
\end{equation}
and
\begin{equation}\label{eq:Psi-def}
    \Psi_\beta(z)
    :=
    {}_2F_1(a_\beta,b_\beta;c;z),
    \qquad 0\leq z<1.
\end{equation}
Since $b_\beta=\overline{a_\beta}$ and $c>0$, the coefficients of the
hypergeometric series satisfy
\[
    \frac{(a_\beta)_k(b_\beta)_k}{(c)_k\,k!}
    =
    \frac{\abs{(a_\beta)_k}^2}{(c)_k\,k!}
    \geq 0.
\]
In particular, $\Psi_\beta$ is real-valued and positive on $[0,1)$.

We now define
\begin{equation}\label{eq:Phi-def}
    \Phi_\beta(t,x)
    :=
    (1+t)^{1-\beta}
    \Psi_\beta\bigl(y(t,x)^2\bigr),
    \qquad (t,x)\in\mathcal C_\ell.
\end{equation}
The dependence on $y^2$ guarantees smoothness at $x=0$, whereas the
factor $(1+t)^{1-\beta}$ allows us to tune the behavior of
$\Phi_\beta$ near the boundary $y=1$. The parameters in \eqref{eq:hyp-parameters} are chosen so that
$\Psi_\beta$ satisfies the hypergeometric equation obtained from the
homogeneous adjoint equation under the scaling \eqref{eq:Phi-def}. The next lemma makes this reduction precise.

\begin{lemma}[Adjoint equation]\label{lem:adjoint}
For every $\beta\in\R$, the function $\Phi_\beta$ defined in \eqref{eq:Phi-def} is real-valued, belongs
to $\mathrm{C}^\infty(\mathcal C_\ell)$, and satisfies
\begin{equation}\label{eq:adjoint}
    \partial_t^2\Phi_\beta
    -
    a(t)^2\Delta\Phi_\beta
    +
    V(t)\Phi_\beta
    =
    0
    \qquad\text{in }\mathcal C_\ell.
\end{equation}
\end{lemma}

\begin{proof}
Put $s:=1+t$, $r:=\abs{x}$, and
\[
    z:=\frac{(1+\ell)^2r^2}{s^{2(1+\ell)}}.
\]
Since $\Psi_\beta$ is analytic for $\abs{z}<1$, its composition with
$r^2$ is smooth at $r=0$. Moreover, $b_\beta=\overline{a_\beta}$, so
the defining power series shows that $\Psi_\beta$, and hence
$\Phi_\beta$, is real-valued for $0\leq z<1$.

Direct differentiation gives
\begin{equation}\label{eq:Phi-t-formula}
    \partial_t\Phi_\beta
    =
    s^{-\beta}
    \Big[
        (1-\beta)\Psi_\beta(z)
        -
        2(1+\ell)z\Psi_\beta'(z)
    \Big],
\end{equation}
\begin{equation}\label{eq:Phi-tt-formula}
    \partial_t^2\Phi_\beta
    =
    s^{-1-\beta}
    \Big[
        \beta(\beta-1)\Psi_\beta(z)
        +
        2(1+\ell)
        \bigl(2\beta+2(1+\ell)-1\bigr)
        z\Psi_\beta'(z)
        +
        4(1+\ell)^2z^2\Psi_\beta''(z)
    \Big],
\end{equation}
and
\begin{equation}\label{eq:laplacian-formula}
    \Delta\Phi_\beta
    =
    (1+\ell)^2
    s^{1-\beta-2(1+\ell)}
    \left[
        4z\Psi_\beta''(z)
        +
        2n\Psi_\beta'(z)
    \right].
\end{equation}
Since $a(t)^2=s^{2\ell}$ and $V(t)=\mu^2s^{-2}$, substitution into
\eqref{eq:adjoint} shows that it is equivalent to
\begin{equation}\label{eq:hyp-ode}
    z(1-z)\Psi_\beta''
    +
    \left[
        \frac n2
        -
        \frac{2\beta+2(1+\ell)-1}{2(1+\ell)}z
    \right]\Psi_\beta'
    -
    \frac{\beta(\beta-1)+\mu^2}{4(1+\ell)^2}
    \Psi_\beta
    =
    0.
\end{equation}
The parameters in \eqref{eq:hyp-parameters} satisfy
\[
    a_\beta+b_\beta+1
    =
    \frac{2\beta+2(1+\ell)-1}{2(1+\ell)}
\]
and
\[
    a_\beta b_\beta
    =
    \frac{(2\beta-1)^2}{16(1+\ell)^2}
    +
    \frac{\omega^2}{4(1+\ell)^2}
    =
    \frac{\beta(\beta-1)+\mu^2}{4(1+\ell)^2}.
\]
Therefore \eqref{eq:hyp-ode} is precisely the hypergeometric equation
satisfied by ${}_2F_1(a_\beta,b_\beta;c;z)$.
\end{proof}

The adjoint equation alone is not sufficient for the critical argument. We
also need positivity and boundedness of the profile, together with a precise
estimate of its time derivative near the boundary of the self-similar
region. These properties are established in the following lemma.

\begin{lemma}[Positivity and boundary behavior]
\label{lem:hyp-bounds}
Let $\beta$ satisfy
\begin{equation}\label{eq:beta-general-range}
    \frac{(1+\ell)(n-2)+1}{2}
    <
    \beta
    <
    \frac{(1+\ell)n+1}{2},
\end{equation}
and set
\begin{equation}\label{eq:delta-beta}
    \delta_\beta
    :=
    c-a_\beta-b_\beta
    =
    \frac{(1+\ell)n+1-2\beta}{2(1+\ell)}
    \in(0,1).
\end{equation}
Then the following statements hold:
\begin{enumerate}[label=$(\roman*)$]

\item
The function $\Psi_\beta$ defined in \eqref{eq:Psi-def} is positive and
strictly increasing on $[0,1)$. It extends continuously to $z=1$ and
satisfies
\begin{equation}\label{eq:Psi-bounds}
    1
    \leq
    \Psi_\beta(z)
    \leq
    \Psi_\beta(1)
    =
    \frac{\Gamma(c)\Gamma(\delta_\beta)}
    {\abs{\Gamma(c-a_\beta)}^2},
    \qquad 0\leq z<1.
\end{equation}

\item
The derivative of $\Psi_\beta$ satisfies
\begin{equation}\label{eq:Psi-prime-asymptotic}
    \lim_{z\to1^-}
    (1-z)^{1-\delta_\beta}\Psi_\beta'(z)
    =
    \frac{\Gamma(c)\Gamma(1-\delta_\beta)}
    {\abs{\Gamma(a_\beta)}^2}
    >0.
\end{equation}
Consequently, there exist constants $0<c_\beta\leq C_\beta$ such that
\begin{equation}\label{eq:Psi-prime-bounds}
    c_\beta(1-z)^{\delta_\beta-1}
    \leq
    \Psi_\beta'(z)
    \leq
    C_\beta(1-z)^{\delta_\beta-1},
    \qquad 0\leq z<1.
\end{equation}

\item
Throughout $\mathcal C_\ell^+$, the self-similar adjoint function satisfies
\begin{equation}\label{eq:Phi-comparison-general}
    (1+t)^{1-\beta}
    \leq
    \Phi_\beta(t,x)
    \leq
    C_\beta(1+t)^{1-\beta}.
\end{equation}

\item
Throughout $\mathcal C_\ell^+$, its time derivative satisfies
\begin{equation}\label{eq:Phi-t-general}
    \abs{\partial_t\Phi_\beta(t,x)}
    \leq
    C_\beta(1+t)^{-\beta}
    \left(
        1-
        \frac{(1+\ell)\abs{x}}{(1+t)^{1+\ell}}
    \right)^{\delta_\beta-1}.
\end{equation}

\end{enumerate}
\end{lemma}

\begin{proof}
Since $b_\beta=\overline{a_\beta}$ and $c=n/2>0$, the defining series can
be written as
\begin{equation}\label{eq:positive-series}
    \Psi_\beta(z)
    =
    \sum_{j=0}^{\infty}
    \frac{(a_\beta)_j(b_\beta)_j}
    {(c)_j\,j!}
    z^j
    =
    \sum_{j=0}^{\infty}
    \frac{\abs{(a_\beta)_j}^2}
    {(c)_j\,j!}
    z^j.
\end{equation}
All coefficients are nonnegative, the constant coefficient is one, and the
coefficient of $z$ is $\abs{a_\beta}^2/c>0$. Therefore, for $0\leq z<1$,
\[
    \Psi_\beta(z)\geq1
    \quad \text{and} \quad
    \Psi_\beta'(z)>0.
\]

\noindent Since $\delta_\beta>0$, Gauss' summation formula
\cite[Eq.~15.4.20]{NIST} gives
\begin{equation}\label{eq:Gauss-sum}
    \Psi_\beta(1)
    =
    \frac{\Gamma(c)\Gamma(\delta_\beta)}
    {\Gamma(c-a_\beta)\Gamma(c-b_\beta)}
    =
    \frac{\Gamma(c)\Gamma(\delta_\beta)}
    {\abs{\Gamma(c-a_\beta)}^2}.
\end{equation}
Together with monotonicity, this proves \eqref{eq:Psi-bounds}.

\smallskip

By the differentiation formula \cite[Eq.~15.5.1]{NIST}, we have
\[
    \Psi_\beta'(z)
    =
    \frac{a_\beta b_\beta}{c} \, 
    {}_2F_1(a_\beta+1,b_\beta+1;c+1;z).
\] Now, let us temporarily set $A:=a_\beta+1$, $B:=b_\beta+1$, and $C:=c+1$; then,
\[
    A+B-C=1-\delta_\beta\in(0,1),
\] so that the boundary formula \cite[Eq.~15.4.23]{NIST} yields
\[
    \lim_{z\to 1^{-}}
    (1-z)^{1-\delta_\beta}\,
    {}_2F_1(A,B;C;z)
    =
    \frac{\Gamma(C)\Gamma(1-\delta_\beta)}
    {\Gamma(A)\Gamma(B)}.
\] Back to our original notation,
we have
\[
    \lim_{z\to 1^{-}}
    (1-z)^{1-\delta_\beta}\Psi_\beta'(z)
    =
    \frac{\Gamma(c)\Gamma(1-\delta_\beta)}
    {\Gamma(a_\beta)\Gamma(b_\beta)}
    =
    \frac{\Gamma(c)\Gamma(1-\delta_\beta)}
    {\abs{\Gamma(a_\beta)}^2},
\] which is \eqref{eq:Psi-prime-asymptotic}.

\smallskip

Next, the function $z \longmapsto (1-z)^{1-\delta_\beta}\Psi_\beta'(z)$ is positive and continuous on $[0,1)$, and
\eqref{eq:Psi-prime-asymptotic} gives a positive finite limit at $z=1$.
It therefore extends to a positive continuous function on $[0,1]$, which
readily implies \eqref{eq:Psi-prime-bounds}.

\smallskip

Estimate \eqref{eq:Phi-comparison-general} follows immediately from
\eqref{eq:Phi-def} and \eqref{eq:Psi-bounds}. Moreover,
\eqref{eq:Phi-t-formula}, \eqref{eq:Psi-bounds}, and
\eqref{eq:Psi-prime-bounds} give
\[
    \abs{\partial_t\Phi_\beta(t,x)}
    \leq
    C_\beta(1+t)^{-\beta}
    (1-z)^{\delta_\beta-1},
\] after possibly redefining the constant $C_\beta$, since $\delta_\beta - 1 < 0$. Next, observe that from the definition of $y$ in \eqref{defofy} and the relation $z=y^2$,
we have \[1-z=(1-y)(1+y).\] Since $1\leq1+y<2$ and $\delta_\beta-1<0$, we obtain
\[
    (1-z)^{\delta_\beta-1}
    \leq
    (1-y)^{\delta_\beta-1}.
\]
This proves \eqref{eq:Phi-t-general}.
\end{proof}

We now choose the parameter $\beta$ as a function of the exponent $p$. At the critical value $p=p_*$, this choice yields the identities that allow us to obtain the logarithmic estimates.

\begin{proposition}\label{prop12}
Let $-1<\ell\leq 0$, $p>1$, and define
\begin{equation}\label{eq:beta-sigma-critical}
 \beta_p:=\frac{(1+\ell)n+1}{2}-\frac{1+\ell}{p} \quad \text{and}
 \quad
 \sigma_p:=(1+\ell)(n-1)
 -\frac{\bigl((1+\ell)n-1\bigr)p}{2}.
\end{equation}
Then,
\begin{equation}\label{eq:beta-general-range2}
 \frac{(1+\ell)(n-2)+1}{2}
 <\beta_p<
 \frac{(1+\ell)n+1}{2}.
\end{equation}
Moreover, if  {$p=\pc>1$ satisfies $\gamma_\alpha(n,\ell;p)=0$, then}
\begin{equation}\label{eq:critical-id-one}
 {2-\alpha-\beta_p+\sigma_p=0},
\end{equation}
as well as
\begin{equation}\label{eq:critical-id-two}
 {\frac{\alpha}{p}+\frac{(1+\ell)n+1-\beta_p}{p'}=1+\frac1p},
 \qquad \text{where } \ p':=\frac{p}{p-1}.
\end{equation}
\end{proposition}

\begin{proof}
For every $p>1$, the definition of $\beta_p$ gives
\[
    \frac{(1+\ell)n+1}{2}-\beta_p
    =
    \frac{1+\ell}{p}
    >0
\]
and
\[
    \beta_p-\frac{(1+\ell)(n-2)+1}{2}
    =
    (1+\ell)\left(1-\frac{1}{p}\right)
    >0.
\]
This proves \eqref{eq:beta-general-range2}.
Moreover, direct computations using \eqref{eq:beta-sigma-critical} yield
\[
     {2-\alpha-\beta_p+\sigma_p=-\frac{\gamma_\alpha(n,\ell;p)}{2p}}
\]
and
\[
     {\frac{\alpha}{p}+\frac{(1+\ell)n+1-\beta_p}{p'}-1-\frac1p=\frac{\gamma_\alpha(n,\ell;p)}{2p^2}}.
\]  {Since $\gamma_\alpha(n,\ell;p)=0$ for $p=\pc$,} these relations give
\eqref{eq:critical-id-one} and \eqref{eq:critical-id-two}.
\end{proof}

We now specialize the self-similar family to the critical exponent $p=p_*$. The resulting profile, recorded in the following corollary, is used as the weight in the nonlinear moment from which the differential inequality underlying the ODE comparison argument is derived.

\begin{corollary}[Critical self-similar adjoint]\label{cor:critical-adjoint}
Let  {$p=\pc>1$ satisfy $\gamma_\alpha(n,\ell;p)=0$} and define
\begin{equation}\label{eq:Phi-p-def}
    \Phi_*:=\Phi_{\beta_{p_*}}.
\end{equation}
Then, throughout $\mathcal C_\ell^+$,
\begin{equation}\label{eq:Phi-p-comparison}
    (1+t)^{1-\beta_{p_*}}
    \leq
    \Phi_*(t,x)
    \leq
    C(1+t)^{1-\beta_{p_*}},
\end{equation}
and
\begin{equation}\label{eq:Phi-p-t}
     {|\partial_t\Phi_*(t,x)|}
    \leq
    C(1+t)^{-\beta_{p_*}}
    \left(
        1-
        \frac{(1+\ell)\abs{x}}{(1+t)^{1+\ell}}
    \right)^{-1/{p'_*}}.
\end{equation}
In particular,
\begin{equation}\label{eq:initial-sign-Phi}
    \Phi_*(0,x)>0,
    \quad\text{for }
    \abs{x}<(1+\ell)^{-1}.
\end{equation}
\end{corollary}

\begin{proof}
By the definition of $\beta_p$ in \eqref{eq:beta-sigma-critical},
\[
    \delta_{\beta_{p_*}}
    =
    \frac{(1+\ell)n+1-2\beta_{p_*}}
    {2(1+\ell)}
    =
    \frac1{p_*}
\quad
\text{and}
\quad
    \delta_{\beta_{p_*}}-1
    =
    -\frac1{p'_*}.
\]
The estimates in \eqref{eq:Phi-p-comparison} and the upper bound in
\eqref{eq:Phi-p-t} follow from \Cref{lem:hyp-bounds}, once $\beta_p$ satisfies \eqref{eq:beta-general-range}.
 {The positivity in \eqref{eq:initial-sign-Phi} follows directly from \eqref{eq:Phi-p-comparison}.}
\end{proof}

The exponent in \eqref{eq:Phi-p-t} is exactly at the borderline of
integrability. After raising the boundary factor to the power $p_*'$, the
resulting singularity is logarithmic. The separation provided by
\eqref{eq:cone-gap} controls this singularity on the propagation region,
leading to the logarithmic estimate established in the next lemma.

\begin{lemma}[Borderline cone integral]\label{lem:cone-log}
Let $\rho(t):=R+A(t)$, and assume that $R<(1+\ell)^{-1}$. Then,
\begin{equation}\label{eq:cone-log}
    \int_{\abs{x}\leq\rho(t)}
    \left(
        1-
        \frac{(1+\ell)\abs{x}}{(1+t)^{1+\ell}}
    \right)^{-1}
    \dd x
    \leq
    C(1+t)^{(1+\ell)n}\log(2+t),
\end{equation}
for every $t\geq0$, where $C$ depends only on $n,\ell$, and $R$.
\end{lemma}

\begin{proof}
Put $s:=1+t$.
Since, as in
\eqref{eq:cone-gap},
\[
    \frac{(1+\ell)\rho(t)}{s^{1+\ell}}
    =
    1-\frac{d_R}{s^{1+\ell}},
\]
polar coordinates and the change of variables
\[
    \zeta:=\frac{(1+\ell)r}{s^{1+\ell}}
\]
give
\begin{equation*}
    \begin{split}
    \int_{\abs{x}\leq\rho(t)}
    \left(
        1-\frac{(1+\ell)\abs{x}}{s^{1+\ell}}
    \right)^{-1}
    \dd x & =
    \abs{\mathbb S^{n-1}}
    \left(
        \frac{s^{1+\ell}}{1+\ell}
    \right)^n
    \int_0^{1-d_R/s^{1+\ell}}
    \frac{\zeta^{n-1}}{1-\zeta}
    \dd \zeta
    \\[0.5em]
    & \leq
    Cs^{(1+\ell)n}
    \int_0^{1-d_R/s^{1+\ell}}
    \frac{1}{1-\zeta}
    \dd \zeta
    \\[0.5em]
    & =
    Cs^{(1+\ell)n}
    \log\left(
        \frac{s^{1+\ell}}{d_R}
    \right)
    \\[0.5em]
    & \leq
    Cs^{(1+\ell)n}\log(2+t). \qedhere
    \end{split}
\end{equation*}
\end{proof}

We conclude by recording that the critical profile is admissible in the
general moment identity. The point is that smoothness is required only in a
neighborhood of the propagation region, not up to nor across the lateral
self-similar boundary.

\begin{remark}\label{rem:moment-Phi-star}
For every finite $\tau$, the cone gap \eqref{eq:cone-gap} implies $K_\tau\subset \subset \mathcal C_\ell$. Hence, $\Phi_{*}$ is smooth in an open neighborhood of $K_\tau$ and satisfies all the hypotheses of \Cref{lem:moment-identity}. No extension across the
self-similar boundary is required.
\end{remark}

\section{Critical differential inequality} \label{sec:critical-inequality}
The purpose of this section is to derive the differential inequalities that
govern the critical blow-up argument. We begin by recording the
nonlinear moment lower bound from \cite{MarconNascimentoSantos2026}, which is valid for every $p>1$ under the
stated assumptions. We then specialize to $p=p_*$ and combine this lower
bound with the critical self-similar adjoint function $\Phi_*$ and the moment
identity from \Cref{lem:moment-identity}. This leads to a nonlinear
differential inequality for a suitable weighted time-integrated functional, which
provides the connection between the PDE and the ODE comparison argument
used in the next section.

Following the notation of \cite{MarconNascimentoSantos2026}, let us define, for every $p>1$, the nonlinear moment
\begin{equation}\label{eq:H-def}
    H(t)
    :=
    \int_{\R^n}|u(t,x)|^p\dd x.
\end{equation}

We present below the estimate needed in the present paper, together with its reformulation in terms of the exponent
$\sigma_p$ defined in \eqref{eq:beta-sigma-critical}.

\begin{proposition}\label{prop:H-lower}
Let $n\in \N$, $-1<\ell\le 0$, $\mu^2>1/4$, $p>1$, $\varepsilon>0$, $f$ and $g$ satisfying \eqref{eq:data-assumptions} with $\supp f\cup\supp g\subset \overline{B_R}$.
Let $u$ be a weak solution with finite speed of propagation of \eqref{eq:main-p} on $[0,T)$. Then, the following statements
hold:
\begin{enumerate}[label=$\mathrm{(\alph*})$]

\item\label{item:previous-lower-bound}
There exist constants $t_1\ge1$ and $C_1>0$, independent of $T$, such that
\begin{equation}\label{eq:H-lower-A}
H(t)
\ge
C_1\,\eps^p\,
(1+t)^{-\frac{\ell p}{2}}
(1+A(t))^{-\frac{n-1}{2}(p-2)}
\end{equation}
for a.e. $t\in[t_1,T)$.

\item\label{item:present-lower-bound}
There exist constants $t_1\ge1$ and $C>0$, independent of $T$, such that
\begin{equation}\label{eq:H-lower}
    H(t)
    \geq
    C\eps^p(1+t)^{\sigma_p}
\end{equation}
for almost every $t\in[t_1,T)$, where $\sigma_p$ is given by \eqref{eq:beta-sigma-critical}.
\end{enumerate}
\end{proposition}

\begin{proof}
Estimate \eqref{eq:H-lower-A} follows directly from
\cite[Proposition~15]{MarconNascimentoSantos2026}.
Indeed, the cited result holds with no restriction on $p>1$. 

It remains to derive \eqref{eq:H-lower}. Since $-1<\ell\leq0$, by the definition
of $A$ in \eqref{eq:a-A-V}, there exists a constant $c>0$ such that $1+A(t)\ge c(1+t)^{1+\ell}$. Using \eqref{eq:H-lower-A} this implies
\begin{align*}
   H(t)
   &\geq
   C_1c\eps^p
   (1+t)^{
       -\ell p/2
       -(1+\ell)\frac{n-1}{2}(p-2)
   }
   \\
   &=:
   C\eps^p
   (1+t)^{
       (1+\ell)(n-1)
       -\frac{((1+\ell)n-1)p}2
   }
   \\
   &=
   C\eps^p(1+t)^{\sigma_p},
\end{align*}
which proves \ref{item:present-lower-bound}.
\end{proof}

From this point onward, we specialize to the critical exponent $p=p_*$. In
this case, the identity
\[
 {2-\alpha-\beta_{p_*}+\sigma_{p_*}=0}
\]
matches the temporal growth in \eqref{eq:H-lower} with the scaling of the
critical self-similar adjoint function $\Phi_*$. We now combine these two
ingredients through the moment identity of \Cref{lem:moment-identity} and
introduce the time-integrated functionals used to derive the critical
differential inequality.

To keep the notation consistent with \Cref{sec:general-adj-identity} and \Cref{sec:hypergeom}, we abbreviate
\[
    \mathcal M_*(t):=M_{\Phi_*}(t),
    \qquad
    \mathcal W_*(t):=W_{\Phi_*}(t),
\]
and define
\begin{equation}\label{eq:G-def}
    G(t)
    :=
    G_{\Phi_*}(t)
    =  {(1+t)^{-\alpha}\int_{\R^n}|u(t,x)|^{p_*}\Phi_*(t,x)\dd x}
\end{equation}
for almost every $t\in(0,T)$.  Here, $\mathcal M_*$ and $\mathcal W_*$ denote, respectively, the moment and
the Wronskian associated with the critical adjoint function $\Phi_*$, while
$G$ is the corresponding weighted nonlinear moment.

We further define
\begin{equation}\label{eq:K-def}
    K(t)
    :=
    \int_0^t(t-s)(2+s)G(s)\dd s
\end{equation}
and
\begin{equation}\label{eq:J-def}
    J(t)
    :=
    \int_0^t(2+s)^{-3}K(s)\dd s.
\end{equation}
Since $G\in L^1_{\loc}([0,T))$, it follows that $K\in W^{2,1}_{\loc}([0,T))$ and $J\in C^2([0,T))$. Moreover,
\begin{equation}\label{eq:K-derivatives}
    K'(t)
    =
    \int_0^t(2+s)G(s)\dd s,
    \qquad
    K''(t)
    =
    (2+t)G(t)
\end{equation}
for almost every $t\in(0,T)$, and
\begin{equation}\label{eq:J-derivative}
    J'(t)
    =
    (2+t)^{-3}K(t).
\end{equation}
Because $\Phi_*>0$, we have $G\geq0$; consequently,
\[
    K(t),\ K'(t),\ J(t),\ J'(t)\geq0
    \qquad\text{for every }t\in[0,T).
\]
We next combine these definitions with the general lower bound from
\Cref{prop:H-lower} to obtain growth estimates for these functions.

\begin{lemma}[Lower bounds for the critical functionals]\label{lem:KJ-lower}
There exist constants $c_0>0$ and $t_2\geq1$, independent of $\eps$ and $T$,
such that
\begin{equation}\label{eq:K-lower}
    K(t)
    \geq
    c_0\eps^{p_*}(2+t)^2,
\end{equation}
\begin{equation}\label{eq:Kprime-initial-lower}
    K'(t)
    \geq
    c_0\eps^{p_*}(2+t),
\end{equation}
\begin{equation}\label{eq:J-lower}
    J(t)
    \geq
    c_0\eps^{p_*}\log(2+t),
\end{equation}
and
\begin{equation}\label{eq:Jprime-lower}
    J'(t)
    \geq
    c_0\eps^{p_*}(2+t)^{-1}
\end{equation}
for every $t\in[t_2,T)$.
\end{lemma}

\begin{proof}
By \eqref{eq:Phi-p-comparison} and \Cref{prop:H-lower}, applied with
$p=p_*$, there exist constants $C>0$ and $t_1\geq1$, independent of
$\eps$ and $T$, such that
\begin{align*}
     {G(t)
    = (1+t)^{-\alpha}\int_{\R^n}|u(t,x)|^{p_*}\Phi_*(t,x)\dd x
    \geq C(1+t)^{1-\alpha-\beta_{p_*}}H(t)
    \geq
    C\eps^{p_*} (1+t)^{1-\alpha-\beta_{p_*}+\sigma_{p_*}}}
\end{align*}
for almost every $t\in[t_1,T)$. By the critical identity
\eqref{eq:critical-id-one},  {$2-\alpha-\beta_{p_*}+\sigma_{p_*}=0$}, and hence
\[
    G(t)
    \geq
    C\eps^{p_*}(1+t)^{-1}.
\]
Consequently,
\begin{equation}\label{eq:weighted-G-lower}
    (2+t)G(t)
    \geq
    C\eps^{p_*}
\end{equation}
for almost every $t\in[t_1,T)$.

Since $K'(t_1)\geq0$, integration of \eqref{eq:weighted-G-lower}, together
with \eqref{eq:K-derivatives}, gives
\[
    K'(t)
    =
    K'(T_1)
    +
    \int_{t_1}^t(2+s)G(s)\dd s
    \geq
    C\eps^{p_*}(t-t_1)
\]
for every $t\in[t_1,T)$. Integrating once more and using $K(t_1)\geq0$, we
obtain
\[
    K(t)=K(t_1)+\frac{C}{2}\eps^{p_*}(t-t_1)^2
    \geq
    \frac{C}{2}\eps^{p_*}(t-t_1)^2.
\]
We may therefore choose a smaller constant $c$ and a $t_2\geq t_1$, depending on $c$ and $t_1$ but independent of $\eps$ and $T$, such that
\begin{equation}\label{eq:preliminary-K-bounds}
    K'(t)
    \geq
    c\eps^{p_*}(2+t),
    \qquad
    K(t)
    \geq
    c\eps^{p_*}(2+t)^2
\end{equation}
for every $t\in[t_2,T)$.

Using \eqref{eq:J-derivative} and the second estimate in
\eqref{eq:preliminary-K-bounds}, we find
\[
    J'(t)
    =
    (2+t)^{-3}K(t)
    \geq
    c\eps^{p_*}(2+t)^{-1}
\]
for every $t\in[t_2,T)$. Since $J(t_2)\geq0$, integration over
$[t_2,t]$ yields
\[
    J(t) \geq J(t)-J(t_2)
    \geq 
    c\eps^{p_*}
    \int_{t_2}^t\frac{\dd s}{2+s}
    =
    c\eps^{p_*}
    \log\left(\frac{2+t}{2+t_2}\right).
\]
Finally, choose $t_3\geq t_2$ so large that
\[
    -\log(2+t_2)
    \geq
    -\frac12\log(2+t_3)
\]
we obtain that 
\[J(t)\ge \frac c2 \eps^{p_*}
    \log(2+t) \qquad\text{for every }t\geq t_3\]
After taking $c_0$ as de minimum of the constants obtained above, and relabelling $t_3$ as $t_2$, all four estimates
\eqref{eq:Kprime-initial-lower}-\eqref{eq:Jprime-lower} follow for every $t\in[t_2,T)$.
\end{proof}

\begin{lemma}[Basic inequality for the critical adjoint]
\label{lem:adjoint-identity}
For every $t\in[0,T)$, one has
\begin{align}\label{eq:adjoint-basic-inequality}
    \int_0^t(t-s)G(s)\dd s
    &\leq
    \int_{\R^n}|u(t,x)|\Phi_*(t,x)\dd x
    +
    2\int_0^t\int_{\R^n}
    |u(s,x)|\,|\partial_s\Phi_*(s,x)|
    \dd x\dd s.
\end{align}
\end{lemma}

\begin{proof}
By \Cref{rem:moment-Phi-star}, the function $\Phi_*$ satisfies the
hypotheses of \Cref{lem:moment-identity}. Applying
\eqref{eq:moment-integrated-general} with $\Phi=\Phi_*$ gives
\begin{align*}
    \int_0^t(t-s)G(s)\dd s
    +t\mathcal W_*(0)
    +\mathcal M_*(0)
    &=
    \mathcal M_*(t)
    -
    2\int_0^t\int_{\R^n}
    u(s,x)\partial_s\Phi_*(s,x)\dd x\dd s.
\end{align*}
The initial moments are
\[
    \mathcal M_*(0)
    =
    \eps\int_{\R^n}f(x)\Phi_*(0,x)\dd x
    \geq0
\]
and
\[
    \mathcal W_*(0)
    =
    \eps\int_{\R^n}
    \Big[
        g(x)\Phi_*(0,x)
        -
        f(x)\partial_t\Phi_*(0,x)
    \Big]\dd x \, 
     {\geq 0}.
\]  {Indeed, putting $z=(1+\ell)^2|x|^2$, formula \eqref{eq:Phi-t-formula} gives}
\[
 {g\Phi_*(0,\cdot)-f\partial_t\Phi_*(0,\cdot)=\bigl(g+(\beta_{p_*}-1)f\bigr)\Psi_{\beta_{p_*}}(z)+2(1+\ell)fz\Psi_{\beta_{p_*}}'(z)\geq0,}
\]
 {by \eqref{eq:data-assumptions} and \Cref{lem:hyp-bounds}.}

Dropping the nonnegative terms $t\mathcal W_*(0)$ and
$\mathcal M_*(0)$, and then taking absolute values on the right-hand side,
we obtain
\begin{align*}
    \int_0^t(t-s)G(s)\dd s
    &\leq
    |\mathcal M_*(t)|
    +
    2\int_0^t\int_{\R^n}
    |u(s,x)|\,|\partial_s\Phi_*(s,x)|
    \dd x\dd s
    \\
    &\leq
    \int_{\R^n}|u(t,x)|\Phi_*(t,x)\dd x
    +
    2\int_0^t\int_{\R^n}
    |u(s,x)|\,|\partial_s\Phi_*(s,x)|
    \dd x\dd s,
\end{align*}
which is \eqref{eq:adjoint-basic-inequality}.
\end{proof}

	

\begin{remark}\label{rem:data-condition}
The sign of the initial adjoint Wronskian in the preceding proof is
determined by the combination $g+(\beta_{p_*}-1)f$. In the unweighted
critical case $\alpha=0$, the critical relation implies
\[
\beta_{p_*}>1.
\]
Hence, for nonnegative initial data $f$ and $g$, one has
$g+(\beta_{p_*}-1)f\geq0$. By the representation of the initial
Wronskian in the proof of \Cref{lem:adjoint-identity} and the positivity properties of the
adjoint profile, it follows that $W_{\Phi_*}(0)\geq0$. Thus, when
$\alpha=0$, the nonnegativity of the initial adjoint Wronskian follows
from the nonnegativity of the initial data.

For a general weight, however, the critical relation
$\gamma_\alpha(n,\ell;p_*)=0$ does not necessarily imply
$\beta_{p_*}>1$. If $\beta_{p_*}<1$, the contribution of the initial
displacement to the Wronskian may have the opposite sign. In this case,
the relevant sign of the data combination is expressed by
\[
g+(\beta_{p_*}-1)f\geq0,
\]
or, equivalently, $g\geq(1-\beta_{p_*})f$. This inequality describes how
the initial velocity offsets the possible negative contribution of the
initial displacement to the Wronskian.
\end{remark}
\normalcolor

To estimate the right-hand side of
\eqref{eq:adjoint-basic-inequality}, we control both weighted terms by the
critical nonlinear moment $G$. The estimate involving $\Phi_*$ follows
from its uniform size inside the propagation region. For
$\partial_t\Phi_*$, the boundary behavior of the critical adjoint produces
a borderline spatial integral and, consequently, the logarithmic factor
appearing below.

\begin{lemma}[Weighted estimates for the critical adjoint]
\label{lem:weighted-u-estimates}
There exists a constant $C>0$, independent of $\eps$ and $T$, such
that, for almost every $t\in(0,T)$,
\begin{equation}\label{eq:uPhi-estimate}
    \int_{\R^n}
    |u(t,x)|\Phi_*(t,x)\dd x
    \leq
    C(1+t)^{1+1/p_*}G(t)^{1/p_*},
\end{equation}
and
\begin{equation}\label{eq:uPhit-estimate}
    \int_{\R^n}
    |u(t,x)|\,|\partial_t\Phi_*(t,x)|
    \dd x
    \leq
    C(1+t)^{1/p_*}
    \bigl(\log(2+t)\bigr)^{1/p_*'}
    G(t)^{1/p_*}.
\end{equation}
\end{lemma}

\begin{proof}
By finite propagation and Hölder's inequality,
\[
    \int_{\R^n}|u(t,x)|\Phi_*(t,x)\dd x
    \leq  {(1+t)^{\alpha/p_*}G(t)^{1/p_*}}
    \left(
        \int_{|x|\leq R+A(t)}
        \Phi_*(t,x)\dd x
    \right)^{1/p_*'}.
\]
Applying \eqref{eq:Phi-comparison-general} with
$\beta=\beta_{p_*}$, and using $R+A(t)\leq C(1+t)^{1+\ell}$, we obtain
\[
    \int_{|x|\leq R+A(t)}
    \Phi_*(t,x)\dd x
    \leq
    C(1+t)^{(1+\ell)n+1-\beta_{p_*}}.
\]
Consequently,
\[
    \int_{\R^n}|u(t,x)|\Phi_*(t,x)\dd x
    \leq
    C(1+t)^{ {\frac{\alpha}{p_*}+\frac{(1+\ell)n+1-\beta_{p_*}}{p_*'}}}
    G(t)^{1/p_*}.
\]
The critical identity \eqref{eq:critical-id-two} now gives
\eqref{eq:uPhi-estimate}.

Similarly, Hölder's inequality yields
\[
    \int_{\R^n}
    |u(t,x)|\,|\partial_t\Phi_*(t,x)|
    \dd x
    \leq  {(1+t)^{\alpha/p_*}G(t)^{1/p_*}}
    \left(
        \int_{|x|\leq R+A(t)}
        |\partial_t\Phi_*(t,x)|^{p_*'}
        \Phi_*(t,x)^{-p_*'/p_*}
        \dd x
    \right)^{1/p_*'}.
\]
From the definition of $\delta_{\beta_{p_*}}$ in \eqref{eq:delta-beta}, we have
\[
    \delta_{\beta_{p_*}}
    =
    \frac{(1+\ell)n+1-2\beta_{p_*}}{2(1+\ell)}
    =
    \frac1{p_*}.
\]
In particular, $p_*'\bigl(\delta_{\beta_{p_*}}-1\bigr)=-1$. Therefore, applying \eqref{eq:Phi-comparison-general} and
\eqref{eq:Phi-t-general} with $\beta=\beta_{p_*}$, we find
\begin{align*}
    |\partial_t\Phi_*(t,x)|^{p_*'}
    \Phi_*(t,x)^{-p_*'/p_*}
    &\leq
    C(1+t)^{
        -\beta_{p_*}p_*'
        +(\beta_{p_*}-1)p_*'/p_*
    }
    \left(
        1-
        \frac{(1+\ell)|x|}{(1+t)^{1+\ell}}
    \right)^{-1}.
\end{align*}
The estimate in \Cref{lem:cone-log} then gives
\begin{align*}
    \int_{|x|\leq R+A(t)}
    |\partial_t\Phi_*(t,x)|^{p_*'}
    \Phi_*(t,x)^{-p_*'/p_*}
    \dd x
    \leq
    C(1+t)^{
        -\beta_{p_*}p_*'
        +(\beta_{p_*}-1)p_*'/p_*
        +(1+\ell)n
    }
    \log(2+t).
\end{align*}
After taking the $p_*'$-th root, the exponent of $1+t$ becomes
\begin{align*}
     {\frac{\alpha}{p_*}-\beta_{p_*}+\frac{\beta_{p_*}-1}{p_*}+\frac{(1+\ell)n}{p_*'}
    = \frac{\alpha}{p_*}+\frac{(1+\ell)n+1-\beta_{p_*}}{p_*'}-1}
    =
    \frac1{p_*},
\end{align*}
where the last equality follows again from
\eqref{eq:critical-id-two}. This proves
\eqref{eq:uPhit-estimate}.
\end{proof}

We next integrate \eqref{eq:adjoint-basic-inequality} over the time
variable. This transforms the linear convolution with $G$ into a quadratic
one which is used to control $J$ from above.

\begin{lemma}[Quadratic critical estimate]
\label{lem:quadratic-critical}
There exists a constant $C>0$, independent of $\eps$ and $T$, such
that, for every $t\in(0,T)$,
\begin{equation}\label{eq:quadratic-critical}
    \frac12\int_0^t(t-s)^2G(s)\dd s
    \leq
    C K'(t)^{1/p_*}
    (2+t)^{1+1/p_*'}
    \bigl(\log(2+t)\bigr)^{1/p_*'}.
\end{equation}
\end{lemma}

\begin{proof}
By \Cref{lem:adjoint-identity,lem:weighted-u-estimates}, for almost every
$r\in(0,T)$,
\begin{equation}\label{eq:int_G}
    \int_0^r(r-s)G(s)\dd s
    \leq
    C(1+r)^{1+1/p_*}G(r)^{1/p_*}
    +
    C\int_0^r
    (1+s)^{1/p_*}
    \bigl(\log(2+s)\bigr)^{1/p_*'}
    G(s)^{1/p_*}\dd s.
\end{equation}
    
Now we integrate this inequality over $r\in(0,t)$. For the left-hand side, Fubini's theorem gives
\[
    \int_0^t\int_0^r(r-s)G(s)\dd s\dd r
    =
    \frac12\int_0^t(t-s)^2G(s)\dd s.
\]

For the first term on the right-hand side of \eqref{eq:int_G}, we use
\[
    (1+r)^{1+1/p_*}G(r)^{1/p_*}
    \leq
    (2+r)\bigl((2+r)G(r)\bigr)^{1/p_*},
\]
with the Hölder's inequality and \eqref{eq:K-derivatives} to obtain
\begin{align*}
    \int_0^t
    (1+r)^{1+1/p_*}G(r)^{1/p_*}\dd r
    &\leq
    \left(
        \int_0^t(2+r)G(r)\dd r
    \right)^{1/p_*}
    \left(
        \int_0^t(2+r)^{p_*'}\dd r
    \right)^{1/p_*'}
    \\
    &\leq
    C K'(t)^{1/p_*}(2+t)^{1+1/p_*'}.
\end{align*}

For the last term on the right-hand side of \eqref{eq:int_G}, Fubini's theorem yields
\begin{align*}
    \int_0^t\int_0^r
    (1+s)^{1/p_*}
    \bigl(\log(2+s)\bigr)^{1/p_*'}
    G(s)^{1/p_*}\dd s\dd r
    &=
    \int_0^t\int_s^t
    (1+s)^{1/p_*}
    \bigl(\log(2+s)\bigr)^{1/p_*'}
    G(s)^{1/p_*}\dd r\dd s\\
    &=\int_0^t
    (t-s)(1+s)^{1/p_*}
    \bigl(\log(2+s)\bigr)^{1/p_*'}
    G(s)^{1/p_*}\dd s
    \\
    &\leq
    \int_0^t
    (t-s)\bigl((2+s)G(s)\bigr)^{1/p_*}
    \bigl(\log(2+s)\bigr)^{1/p_*'}
    \dd s.
\end{align*}
Another application of Hölder's inequality gives
\begin{align*}
    \int_0^t
    (t-s)\bigl((2+s)G(s)\bigr)^{1/p_*}
    \bigl(\log(2+s)\bigr)^{1/p_*'}
    \dd s
    \leq
    K'(t)^{1/p_*}
    \left(
        \int_0^t
        (t-s)^{p_*'}\log(2+s)\dd s
    \right)^{1/p_*'}.
\end{align*}
Since
\[
    \int_0^t
    (t-s)^{p_*'}\log(2+s)\dd s
    \leq
    \frac{t^{p_*'+1}}{p_*'+1}\log(2+t)
    \leq 
    (2+t)^{p_*'+1}\log(2+t),
\]
it follows that
\begin{align*}
    &\int_0^t\int_0^r
    (1+s)^{1/p_*}
    \bigl(\log(2+s)\bigr)^{1/p_*'}
    G(s)^{1/p_*}\dd s\dd r
    \leq
    C K'(t)^{1/p_*}
    (2+t)^{1+1/p_*'}
    \bigl(\log(2+t)\bigr)^{1/p_*'}.
\end{align*}
Finally, since $\log(2+t)\geq\log 2>0$, the contribution of the first
term is bounded by the same expression. This proves
\eqref{eq:quadratic-critical}.
\end{proof}

The quadratic convolution obtained in \Cref{lem:quadratic-critical} admits
an exact representation in terms of $J$.

\begin{lemma}[Exact representation of $J$]
\label{lem:J-G}
For every $t\in[0,T)$,
\begin{equation}\label{eq:J-G}
    (2+t)^2J(t)
    =
    \frac12\int_0^t(t-s)^2G(s)\dd s.
\end{equation}
\end{lemma}

\begin{proof}
By the definitions of $J$ and $K$ in \eqref{eq:K-def} and \eqref{eq:J-def},
\[
    J(t)
    =
    \int_0^t
    \frac{1}{(2+r)^3}
    \int_0^r
    (r-s)(2+s)G(s)\dd s\dd r.
\]
Again, Fubini's theorem plus a direct calculation gives
\begin{align*}
J(t) &=
    \int_0^t
    (2+s)G(s)
    \left(
        \int_s^t
        \frac{r-s}{(2+r)^3}\dd r
    \right)
    \dd s \\
    &= \int_0^t
    (2+s)G(s)
        \frac{(t-s)^2}
    {2(2+s)(2+t)^2}
    \dd s \\
    &=\frac{1}{2(2+t)^2}
    \int_0^t(t-s)^2G(s)\dd s,
\end{align*}
which is equivalent to \eqref{eq:J-G}.
\end{proof}

Combining the quadratic estimate in \Cref{lem:quadratic-critical} with the
exact representation in \Cref{lem:J-G}, and then using the relation between
$K$ and $J'$, yields the closed differential inequality needed in the
blow-up argument.

\begin{proposition}[Critical differential inequality]
\label{prop:critical-differential}
There exists a constant $C_0>0$, independent of $\eps$ and $T$, such
that, for every $t\in(0,T)$,
\begin{equation}\label{eq:J-critical-differential}
    (2+t)^2J''(t)
    +
    3(2+t)J'(t)
    \geq
    C_0\bigl(\log(2+t)\bigr)^{1-p_*}
    J(t)^{p_*}.
\end{equation}
\end{proposition}

\begin{proof}
By \Cref{lem:quadratic-critical,lem:J-G},
\[
    (2+t)^2J(t)
    \leq
    C K'(t)^{1/p_*}
    (2+t)^{1+1/p_*'}
    \bigl(\log(2+t)\bigr)^{1/p_*'}.
\]
Raising this inequality to the power $p_*$, setting $C_0=C^{-p_*}$ and rearranging, we obtain
\begin{equation}\label{eq:Kprime-J}
    K'(t)
    \geq
    C_0(2+t)
    \bigl(\log(2+t)\bigr)^{1-p_*}
    J(t)^{p_*}.
\end{equation}

On the other hand, from \eqref{eq:J-derivative} we have
\[
    K(t)=(2+t)^3J'(t),
\]
and hence
\[
    K'(t)
    =
    (2+t)^3J''(t)
    +
    3(2+t)^2J'(t).
\]
Replacing this expression into \eqref{eq:Kprime-J} and dividing by $2+t$ proves
\eqref{eq:J-critical-differential}.
\end{proof}

\section{An ODE comparison lemma}\label{sec:ode}
The following comparison principle for ordinary differential inequalities was first established in \cite{ZhouHan2014}. It was
subsequently used to prove finite-time blow-up at critical exponents and to
derive the corresponding lifespan estimates in works such as
\cite{IkedaSobajima2018,PalmieriTu2019}. We include its proof for the
convenience of the reader and to make the presentation more complete and
self-contained.

\begin{lemma}[ODE lifespan estimate]
\label{lem:critical-ODE}
Let $p>1$ and $C_0,c_0,\tau_0>0$. There exist constants
$\eta_0\in(0,1]$ and $C>0$, depending only on
$p,C_0,c_0$, and $\tau_0$, with the following property.

Let $\tau_*\in(\tau_0,+\infty]$, $0<\eta\leq\eta_0$, and suppose that $\mathcal J\in W_{\loc}^{2,1}([\tau_0,\tau_*))$
satisfies
\begin{numcases}{}
  \mathcal J''(\tau)+2\mathcal J'(\tau) \geq C_0\tau^{1-p}\mathcal J(\tau)^p, & a.e.\ $\tau\in(\tau_0,\tau_*)$ \label{eq:ODE-inequality} \\
  \mathcal J(\tau)\geq c_0\eta\tau,\; \mathcal J'(\tau)\geq c_0\eta, & for all\ $\tau\in[\tau_0,\tau_*)$ \label{eq:ODE-lower-data}
\end{numcases}

Then
\begin{equation}\label{eq:ODE-lifespan}
    \tau_*
    \leq
    C\eta^{-(p-1)}.
\end{equation}
\end{lemma}

\begin{proof}
Set $L:=\eta^{-(p-1)}$, $\delta:=L^{-1}=\eta^{p-1}\leq1$, and define
\[
    \Lambda(x)
    :=
    \frac{\mathcal J(Lx)}{\eta L},
    \qquad
    \frac{\tau_0}{L}\leq x<\frac{\tau_*}{L}.
\]
Since $\mathcal J'(Lx)=\eta\Lambda'(x)$ and $\mathcal J''(Lx)=\frac{\eta}{L}\Lambda''(x)$, inequalities \eqref{eq:ODE-inequality}-\eqref{eq:ODE-lower-data} become
\begin{numcases}{}
  \delta\Lambda''(x)+2\Lambda'(x)\geq C_0x^{1-p}\Lambda(x)^p, & a.e.\ $x\in\left(\frac{\tau_0}{L},\frac{\tau_*}{L}\right)$ \label{eq:scaled-ODE} \\
  \Lambda(x)\geq c_0x,\; \Lambda'(x)\geq c_0, & for all\ $x\in\left[\frac{\tau_0}{L},\frac{\tau_*}{L}\right)$ \label{eq:scaled-lower}
\end{numcases}

We next construct a singular subsolution. Fix $h\geq2$, to be chosen
below, and set
\[
    q:=\frac{2}{p-1},
    \qquad
    b_h:=C_0(2h)^{1-p},
\]
and
\[
    M_h
    :=
    \left(
        \frac{q(q+1)+2qh}{b_h}
    \right)^{1/(p-1)}.
\]
Define
\[
    z_h(x):=M_h(2h-x)^{-q},
    \qquad h\leq x<2h.
\]
Since $pq=q+2$, we have
\begin{align*}
    \delta z_h''(x)+2z_h'(x)
    &=
    M_h(2h-x)^{-q-2}
    \left[
        \delta q(q+1)+2q(2h-x)
    \right]
    \\
    &\leq
    M_h(2h-x)^{-q-2}
    \left[
        q(q+1)+2qh
    \right]
    \\
    &=
    b_hz_h(x)^p.
\end{align*}
Thus $z_h$ is a subsolution independently of $0<\delta\leq1$.

As $h\to\infty$, note that
\begin{equation*}
    \frac{z_h(h)}{h}+z_h'(h)
    =
    (1+q)M_hh^{-q-1}
    =
    \left(
        \frac{q(q+1)+2qh}{C_02^{1-p}h^2}
    \right)^{1/(p-1)}\to 0.
\end{equation*}
Therefore, we may fix $h\geq2$, depending only on $p,C_0$, and
$c_0$, sufficiently large that
\begin{equation}\label{eq:subsolution-initial}
    z_h(h)<c_0h,
    \qquad
    z_h'(h)<c_0.
\end{equation}

For this fixed $h$, let
\[
    \eta_0
    :=
    \min\left\{
        1,
        \left(\frac{h}{\tau_0}\right)^{1/(p-1)}
    \right\}.
\]
If $0<\eta\leq\eta_0$, then
\[
    \frac{\tau_0}{L}
    =
    \tau_0\eta^{p-1}
    \leq h.
\]
If $\tau_*/L\leq h$, then $\tau_* \leq hL = h\eta^{-(p-1)}$, and the conclusion follows. 
\\
Assume therefore that $\tau_*/L>h$. By \eqref{eq:scaled-lower} and \eqref{eq:subsolution-initial}, we have $\Lambda(h)>z_h(h)$ and $\Lambda'(h)>z_h'(h)$.

We claim that
\begin{equation}\label{eq:Lambda-z-comparison}
    \Lambda(x)\geq z_h(x)
    \qquad
    \text{for }
    h\leq x<
    \min\left\{
        2h,\frac{\tau_*}{L}
    \right\}.
\end{equation}
Set    $X:=
    \min\left\{
        2h,\frac{\tau_*}{L}
    \right\}$ and
    $w:=\Lambda-z_h$. By the choice of $h$, we have
\[
    w(h)>0
    \qquad\text{and}\qquad
    w'(h)>0.
\]
Suppose, by contradiction, that $w$ vanishes somewhere in $(h,X)$,
and let $x_0\in(h,X)$ be its first zero. Then $w(x)>0$ for $h\leq x<x_0$ and $w(x_0)=0$. In particular, $\Lambda(x)\geq z_h(x)>0$ on $[h,x_0]$. Since
$r\mapsto r^p$ is increasing on $[0,\infty)$ and $(h,X)\subset \left(\frac{\tau_0}{L},\frac{\tau_*}{L}\right)$, we obtain
\begin{align*}
    \delta w''(x)+2w'(x) &\geq C_0x^{1-p}\Lambda(x)^p - b_hz_h(x)^p\\
    &\geq
    b_h\bigl(\Lambda(x)^p-z_h(x)^p\bigr)\\
    &\geq0
\end{align*}
for almost every $x\in(h,x_0)$, where we used that $C_0x^{1-p}\geq b_h$ since $x\le 2h$. Therefore,
\[
    \left(e^{2x/\delta}w'(x)\right)'
    =
    \frac1\delta e^{2x/\delta}
    \bigl(\delta w''(x)+2w'(x)\bigr)
    \geq0
\]
for almost every $x\in(h,x_0)$. Integrating from $h$ to $x<x_0$
gives
\[
    e^{2x/\delta}w'(x)
    \geq
    e^{2h/\delta}w'(h)>0,
\]
and hence $w'(x)>0$ on $[h,x_0)$. Thus $w$ is strictly increasing
there, so $w(x_0)\geq w(h)>0$, contradicting the definition of $x_0$. Consequently, $w$ cannot vanish
in $(h,X)$, and \eqref{eq:Lambda-z-comparison} follows.

Finally, from the definition of $z_h$ we have that $z_h(x)\longrightarrow+\infty$ as $x\to2h^-$.

Suppose that $\tau_*/L>2h$. Since $\Lambda\in W_{\loc}^{2,1}  \left( \left[\frac{\tau_0}{L},\frac{\tau_*}{L}\right) \right)$,
and $2h$ is an interior point of this interval, $\Lambda$ is continuous
and bounded in a neighborhood of $2h$. On the other hand,
\eqref{eq:Lambda-z-comparison} gives
\[
    \Lambda(x)\geq z_h(x)\longrightarrow+\infty
    \qquad\text{as }x\to2h^-,
\]
which is a contradiction. Therefore, $\frac{\tau_*}{L}\leq2h$ and hence
\[
    \tau_*
    \leq
    2hL
    =
    2h\eta^{-(p-1)}.
\]
This proves \eqref{eq:ODE-lifespan}.
\end{proof}

\section{Proof of the main theorem}
\label{sec:proof-main}

The \Cref{sec:critical-inequality} reduces the PDE argument to two estimates for the
functional $J$: the critical differential inequality in
\Cref{prop:critical-differential}, valid throughout the lifespan, and the
lower bounds in \Cref{lem:KJ-lower}, valid after a fixed time independent
of $\eps$. A logarithmic change of variables transforms these estimates
into the hypotheses of \Cref{lem:critical-ODE}, from which both finite-time
blow-up and the critical exponential lifespan bound follow.

\begin{proof}[Proof of \Cref{thm:main}]
Let $p=p_*$, and let $u$ be a weak solution with finite speed of
propagation of \eqref{eq:main-p} on its maximal existence interval
$[0,T_\eps)$.

By \Cref{prop:critical-differential}, there exists $C_0>0$, independent
of $\eps$ and $T_\eps$, such that
\begin{equation}\label{eq:J-differential-final}
    (2+t)^2J''(t)
    +
    3(2+t)J'(t)
    \geq
    C_0\bigl(\log(2+t)\bigr)^{1-p_*}J(t)^{p_*}, \qquad \forall\, t\in(0,T_\eps).
\end{equation}
 Moreover, by
\Cref{lem:KJ-lower}, there exist $c_0>0$ and $t_2\geq1$, independent
of $\eps$ and $T_\eps$, such that
\begin{equation}\label{eq:J-lower-final}
    J(t)\geq c_0\eps^{p_*}\log(2+t),
    \qquad
    (2+t)J'(t)\geq c_0\eps^{p_*}, \qquad \forall\, t\in[t_2,T_\eps),
\end{equation}
if $t_2<T_\eps$. In the case where $t_2\geq T_\eps$, we trivially obtain the lifespan estimate \eqref{eq:lifespan-exp} by noting that, for any $0<\eps\le \eps_0:=1$, we have
\[T_\eps<t_2<\exp\left(\log(1+t_2)\right)\le \exp\left(\log(1+t_2)\eps^{-p_*(p_*-1)}\right)\]
and the estimate holds for constant $C=\log(1+t_2)$, which is independent of $\eps$.

Thus, for $t_2<T_\eps$, define the change of variable $\tau = \log(2+t)$, $t\in[t_2,T_\eps)$ and set $\tau_0:= \log(2+t_2)$ and $\tau_\eps=\log(2+T_\eps)$, where $\tau_\eps=+\infty$ if $T_\eps=+\infty$. Finally, define the functional 
\[
    \mathcal J(\tau):=J(e^\tau-2),
    \qquad
    \tau\in[\tau_0,\tau_\eps),
\]
where $\tau_\eps=+\infty$ if $T_\eps=+\infty$. Since
$J\in C^2([0,T_\eps))$, we have $\mathcal J\in
    W_{\loc}^{2,1}([\tau_0,\tau_\eps))$.
    
Writing $t=e^\tau-2$, we compute
\[
    \mathcal J'(\tau)
    =
    (2+t)J'(t)
\]
and
\[
    \mathcal J''(\tau)
    =
    (2+t)^2J''(t)+(2+t)J'(t).
\]
Therefore,
\[
    \mathcal J''(\tau)+2\mathcal J'(\tau)
    =
    (2+t)^2J''(t)+3(2+t)J'(t).
\]
Since $\tau=\log(2+t)$, the differential inequality
\eqref{eq:J-differential-final} and the inequalities \eqref{eq:J-lower-final} become
\begin{numcases}{}
  \mathcal J''(\tau)+2\mathcal J'(\tau)
    \geq
    C_0\tau^{1-p_*}\mathcal J(\tau)^{p_*}, & a.e.\ $\tau\in(\tau_0,\tau_\eps)$ \label{eq:mathcalJ-ODE-final} \\
  \mathcal J(\tau)\geq c_0\eps^{p_*}\tau,\; \mathcal J'(\tau)\geq c_0\eps^{p_*}, & for all\ $\tau\in[\tau_0,\tau_\eps)$ \label{eq:mathcalJ-lower-final}
\end{numcases}

We first prove finite-time blow-up for arbitrary $\eps>0$.
Let $\eta_0\in(0,1]$ and set $\eta:=\min\{\eps^{p_*},\eta_0\}$. Hence, by \eqref{eq:mathcalJ-lower-final},
\[
    \mathcal J(\tau)\geq c_0\eta\tau,
    \qquad
    \mathcal J'(\tau)\geq c_0\eta
\]
for every $\tau\in[\tau_0,\tau_\eps)$.
Together with \eqref{eq:mathcalJ-ODE-final}, these estimates show that all the hypotheses of \Cref{lem:critical-ODE} are satisfied with $p=p_*$ and $\tau_*=\tau_\eps$. Therefore, there exist $C>0$ independent of $\eta$ such that
\[
    \tau_\eps
    \leq
    C\eta^{-(p_*-1)}
    <
    +\infty.
\]
Consequently, $T_\eps=e^{\tau_\eps}-2<+\infty$. Thus every solution under consideration has finite maximal existence time for every $\eps>0$.

It remains to prove the quantitative lifespan estimate. Define $\eps_0:=\eta_0^{1/p_*}\in(0,1]$, and let $0<\eps\leq\eps_0$.

Defining $\eta:=\eps^{p_*}$ then $0<\eta\leq\eta_0$ and the estimates \eqref{eq:mathcalJ-ODE-final} and
\eqref{eq:mathcalJ-lower-final} satisfy the hypotheses of
\Cref{lem:critical-ODE}. Consequently,
\[
    \log(2+T_\eps)
    =
    \tau_\eps
    \leq
    C\eta^{-(p_*-1)}
    \leq
    C\bigl(\eps^{p_*}\bigr)^{-(p_*-1)}
    =
    C\eps^{-p_*(p_*-1)}.
\]
Finally, it follows that
\[
    T_\eps
    \leq
    \exp\!\left(
        C\eps^{-p_*(p_*-1)}
    \right),
\]
which proves \eqref{eq:lifespan-exp}.
\end{proof}

\newpage

\bibliographystyle{amsplain}
\bibliography{ref}

@article{John1979,
  author  = {John, Fritz},
  title   = {Blow-up of solutions of nonlinear wave equations in three space
             dimensions},
  journal = {Manuscripta Math.},
  volume  = {28},
  number  = {1--3},
  pages   = {235--268},
  year    = {1979},
  doi     = {10.1007/BF01647974}
}

@article{Glassey1981a,
  author  = {Glassey, Robert T.},
  title   = {Finite-time blow-up for solutions of nonlinear wave equations},
  journal = {Math. Z.},
  volume  = {177},
  number  = {3},
  pages   = {323--340},
  year    = {1981},
  doi     = {10.1007/BF01162066}
}

@article{Sideris1984,
  author  = {Sideris, Thomas C.},
  title   = {Nonexistence of global solutions to semilinear wave equations in
             high dimensions},
  journal = {J. Differential Equations},
  volume  = {52},
  number  = {3},
  pages   = {378--406},
  year    = {1984},
  doi     = {10.1016/0022-0396(84)90169-4}
}

@article{Schaeffer1985,
  author  = {Schaeffer, Jack},
  title   = {The equation {$u_{tt}-\Delta u=|u|^p$} for the critical value of {$p$}},
  journal = {Proc. Roy. Soc. Edinburgh Sect. A},
  volume  = {101},
  number  = {1--2},
  pages   = {31--44},
  year    = {1985},
  doi     = {10.1017/S0308210500026135}
}

@article{GeorgievLindbladSogge1997,
  author  = {Georgiev, Vladimir and Lindblad, Hans and Sogge, Christopher D.},
  title   = {Weighted {S}trichartz estimates and global existence for
             semilinear wave equations},
  journal = {Amer. J. Math.},
  volume  = {119},
  number  = {6},
  pages   = {1291--1319},
  year    = {1997},
  doi     = {10.1353/ajm.1997.0038}
}

@article{YordanovZhang2006,
  author  = {Yordanov, Borislav T. and Zhang, Qi S.},
  title   = {Finite time blow up for critical wave equations in high dimensions},
  journal = {J. Funct. Anal.},
  volume  = {231},
  number  = {2},
  pages   = {361--374},
  year    = {2006},
  doi     = {10.1016/j.jfa.2005.03.012}
}

@article{Zhou1992,
  author  = {Zhou, Yi},
  title   = {Blow up of classical solutions to {$\Box u=|u|^{1+\alpha}$} in
             three space dimensions},
  journal = {J. Partial Differential Equations},
  volume  = {5},
  number  = {3},
  pages   = {21--32},
  year    = {1992}
}

@article{Zhou2007,
  author  = {Zhou, Yi},
  title   = {Blow up of solutions to semilinear wave equations with critical
             exponent in high dimensions},
  journal = {Chin. Ann. Math. Ser. B},
  volume  = {28},
  number  = {2},
  pages   = {205--212},
  year    = {2007},
  doi     = {10.1007/s11401-005-0205-x}
}

@article{TakamuraWakasa2011,
  author  = {Takamura, Hiroyuki and Wakasa, Kyouhei},
  title   = {The sharp upper bound of the lifespan of solutions to critical
             semilinear wave equations in high dimensions},
  journal = {J. Differential Equations},
  volume  = {251},
  number  = {4--5},
  pages   = {1157--1171},
  year    = {2011},
  doi     = {10.1016/j.jde.2011.03.024}
}

@article{ZhouHan2014,
  author  = {Zhou, Yi and Han, Wei},
  title   = {Life-span of solutions to critical semilinear wave equations},
  journal = {Comm. Partial Differential Equations},
  volume  = {39},
  number  = {3},
  pages   = {439--451},
  year    = {2014},
  doi     = {10.1080/03605302.2013.863914}
}

@article{Yagdjian2004,
  author  = {Yagdjian, Karen},
  title   = {A note on the fundamental solution for the {T}ricomi-type equation
             in the hyperbolic domain},
  journal = {J. Differential Equations},
  volume  = {206},
  number  = {1},
  pages   = {227--252},
  year    = {2004},
  doi     = {10.1016/j.jde.2004.07.028}
}

@article{Yagdjian2006,
  author  = {Yagdjian, Karen},
  title   = {Global existence for the {$n$}-dimensional semilinear
             {T}ricomi-type equations},
  journal = {Comm. Partial Differential Equations},
  volume  = {31},
  number  = {4--6},
  pages   = {907--944},
  year    = {2006},
  doi     = {10.1080/03605300500361511}
}

@article{Yagdjian2007a,
  author  = {Yagdjian, Karen},
  title   = {The self-similar solutions of the one-dimensional semilinear
             {T}ricomi-type equations},
  journal = {J. Differential Equations},
  volume  = {236},
  number  = {1},
  pages   = {82--115},
  year    = {2007},
  doi     = {10.1016/j.jde.2007.01.017}
}

@article{Yagdjian2007b,
  author  = {Yagdjian, Karen},
  title   = {Self-similar solutions of semilinear wave equation with variable
             speed of propagation},
  journal = {J. Math. Anal. Appl.},
  volume  = {336},
  number  = {2},
  pages   = {1259--1286},
  year    = {2007},
  doi     = {10.1016/j.jmaa.2007.03.065}
}

@article{HeWittYin2017a,
  author  = {He, Daoyin and Witt, Ingo and Yin, Huicheng},
  title   = {On the global solution problem for semilinear generalized
             {T}ricomi equations, {I}},
  journal = {Calc. Var. Partial Differential Equations},
  volume  = {56},
  number  = {2},
  pages   = {Paper No. 21},
  year    = {2017},
  doi     = {10.1007/s00526-017-1125-9}
}

@article{HeWittYin2017,
  author  = {He, Daoyin and Witt, Ingo and Yin, Huicheng},
  title   = {On semilinear {T}ricomi equations with critical exponents or in
             two space dimensions},
  journal = {J. Differential Equations},
  volume  = {263},
  number  = {12},
  pages   = {8102--8137},
  year    = {2017},
  doi     = {10.1016/j.jde.2017.08.033}
}

@article{HeWittYin2020,
  author  = {He, Daoyin and Witt, Ingo and Yin, Huicheng},
  title   = {On the {S}trauss index of semilinear {T}ricomi equation},
  journal = {Commun. Pure Appl. Anal.},
  volume  = {19},
  number  = {10},
  pages   = {4817--4838},
  year    = {2020},
  doi     = {10.3934/cpaa.2020213}
}

@misc{LinTu2019,
  author       = {Lin, Jiayun and Tu, Ziheng},
  title        = {Lifespan of semilinear generalized {T}ricomi equation with
                  {S}trauss type exponent},
  year         = {2019},
  eprint       = {1903.11351},
  archivePrefix= {arXiv},
  primaryClass = {math.AP},
  note         = {Preprint}
}

@article{Palmieri2025,
  author  = {Palmieri, Alessandro},
  title   = {On the critical exponent for the semilinear
             {E}uler--{P}oisson--{D}arboux--{T}ricomi equation with power
             nonlinearity},
  journal = {J. Differential Equations},
  volume  = {437},
  pages   = {Paper No. 113309},
  year    = {2025},
  doi     = {10.1016/j.jde.2025.113309}
}

@article{LaiPalmieriTakamura2026,
  author  = {Lai, Ning-An and Palmieri, Alessandro and Takamura, Hiroyuki},
  title   = {A blow-up result for the semilinear
             {E}uler--{P}oisson--{D}arboux--{T}ricomi equation with critical
             power nonlinearity},
  journal = {J. Math. Anal. Appl.},
  volume  = {553},
  number  = {1},
  pages   = {Paper No. 129835},
  year    = {2026},
  doi     = {10.1016/j.jmaa.2025.129835}
}

@article{DLR2015,
  author  = {D'Abbicco, Marcello and Lucente, Sandra and Reissig, Michael},
  title   = {A shift in the {S}trauss exponent for semilinear wave equations
             with a not effective damping},
  journal = {J. Differential Equations},
  volume  = {259},
  number  = {10},
  pages   = {5040--5073},
  year    = {2015},
  doi     = {10.1016/j.jde.2015.06.018}
}

@article{DAbbicco2015,
  author  = {D'Abbicco, Marcello},
  title   = {The threshold of effective damping for semilinear wave equations},
  journal = {Math. Methods Appl. Sci.},
  volume  = {38},
  number  = {6},
  pages   = {1032--1045},
  year    = {2015},
  doi     = {10.1002/mma.3126}
}

@article{NascimentoPalmieriReissig2017,
  author  = {Nunes do Nascimento, Wanderley and Palmieri, Alessandro and
             Reissig, Michael},
  title   = {Semi-linear wave models with power non-linearity and
             scale-invariant time-dependent mass and dissipation},
  journal = {Math. Nachr.},
  volume  = {290},
  number  = {11--12},
  pages   = {1779--1805},
  year    = {2017},
  doi     = {10.1002/mana.201600069}
}

@article{PalmieriReissig2018,
  author  = {Palmieri, Alessandro and Reissig, Michael},
  title   = {Semi-linear wave models with power non-linearity and
             scale-invariant time-dependent mass and dissipation, {II}},
  journal = {Math. Nachr.},
  volume  = {291},
  number  = {11--12},
  pages   = {1859--1892},
  year    = {2018},
  doi     = {10.1002/mana.201700144}
}

@article{PalmieriReissig2019,
  author  = {Palmieri, Alessandro and Reissig, Michael},
  title   = {A competition between {F}ujita and {S}trauss type exponents for
             blow-up of semi-linear wave equations with scale-invariant damping
             and mass},
  journal = {J. Differential Equations},
  volume  = {266},
  number  = {2--3},
  pages   = {1176--1220},
  year    = {2019},
  doi     = {10.1016/j.jde.2018.07.061}
}

@article{PalmieriTu2019,
  author  = {Palmieri, Alessandro and Tu, Ziheng},
  title   = {Lifespan of semilinear wave equation with scale invariant
             dissipation and mass and sub-{S}trauss power nonlinearity},
  journal = {J. Math. Anal. Appl.},
  volume  = {470},
  number  = {1},
  pages   = {447--469},
  year    = {2019},
  doi     = {10.1016/j.jmaa.2018.10.015}
}

@article{IkedaSobajima2018,
  author  = {Ikeda, Masahiro and Sobajima, Motohiro},
  title   = {Life-span of solutions to semilinear wave equation with
             time-dependent critical damping for specially localized initial
             data},
  journal = {Math. Ann.},
  volume  = {372},
  number  = {3--4},
  pages   = {1017--1040},
  year    = {2018},
  doi     = {10.1007/s00208-018-1664-1}
}

@article{IkedaSobajimaWakasa2019,
  author  = {Ikeda, Masahiro and Sobajima, Motohiro and Wakasa, Kyouhei},
  title   = {Blow-up phenomena of semilinear wave equations and their weakly
             coupled systems},
  journal = {J. Differential Equations},
  volume  = {267},
  number  = {9},
  pages   = {5165--5201},
  year    = {2019},
  doi     = {10.1016/j.jde.2019.05.029}
}

@misc{Hamza2026,
  author       = {Hamza, Mohamed Ali},
  title        = {On the blow-up of solutions to scale-invariant wave equations
                  with damping and mass: beyond the positive discriminant
                  restriction},
  year         = {2026},
  eprint       = {2604.06478},
  archivePrefix= {arXiv},
  primaryClass = {math.AP},
  note         = {Preprint}
}

@article{MarconNascimentoSantos2026,
  author  = {Marcon, Diego and Nunes do Nascimento, Wanderley and
             Santos, Matheus C.},
  title   = {Blow-up for a semilinear {T}ricomi-type equation with
             scale-invariant mass in the oscillatory regime},
  journal = {Math. Ann.},
  volume  = {396},
  number  = {40},
  year    = {2026},
  doi     = {10.1007/s00208-026-03579-2}
}

@book{NIST,
  editor    = {Olver, Frank W. J. and Lozier, Daniel W. and
               Boisvert, Ronald F. and Clark, Charles W.},
  title     = {{NIST} Handbook of Mathematical Functions},
  publisher = {Cambridge University Press},
  address   = {Cambridge},
  year      = {2010}
}

@article{TodorovaYordanov2001,
  author  = {Todorova, Grozdena and Yordanov, Borislav},
  title   = {Critical exponent for a nonlinear wave equation with damping},
  journal = {J. Differential Equations},
  volume  = {174},
  number  = {2},
  pages   = {464--489},
  year    = {2001},
  doi     = {10.1006/jdeq.2000.3933}
}

@article{Zhang2001,
  author  = {Zhang, Qi S.},
  title   = {A blow-up result for a nonlinear wave equation with damping: the
             critical case},
  journal = {C. R. Acad. Sci. Paris S\'er. I Math.},
  volume  = {333},
  number  = {2},
  pages   = {109--114},
  year    = {2001},
  doi     = {10.1016/S0764-4442(01)01999-1}
}

@misc{HeWittYin2016II,
  author       = {He, Daoyin and Witt, Ingo and Yin, Huicheng},
  title        = {On the global solution problem of semilinear generalized
                  {T}ricomi equations, {II}},
  year         = {2016},
  eprint       = {1611.07606},
  archivePrefix= {arXiv},
  primaryClass = {math.AP},
  note         = {Preprint}
}

@misc{HeWittYin2018,
  author       = {He, Daoyin and Witt, Ingo and Yin, Huicheng},
  title        = {On semilinear {T}ricomi equations in one space dimension},
  year         = {2018},
  eprint       = {1810.12748},
  archivePrefix= {arXiv},
  primaryClass = {math.AP},
  note         = {Preprint}
}

@misc{LiGuo2025,
  author       = {Li, Yuequn and Guo, Fei},
  title        = {Critical curve for weakly coupled system of semilinear
                  {E}uler--{P}oisson--{D}arboux--{T}ricomi equations},
  year         = {2025},
  eprint       = {2511.08084},
  archivePrefix= {arXiv},
  primaryClass = {math.AP},
  note         = {Preprint}
}

@incollection{PalmieriOdd2019,
  author    = {Palmieri, Alessandro},
  title     = {Global existence results for a semilinear wave equation with
               scale-invariant damping and mass in odd space dimension},
  booktitle = {New Tools for Nonlinear {PDE}s and Application},
  series    = {Trends in Mathematics},
  publisher = {Birkh\"auser/Springer},
  address   = {Cham},
  pages     = {305--369},
  year      = {2019},
  doi       = {10.1007/978-3-030-10937-0_12}
}

@article{PalmieriEven2019,
  author  = {Palmieri, Alessandro},
  title   = {A global existence result for a semilinear scale-invariant wave
             equation in even dimension},
  journal = {Math. Methods Appl. Sci.},
  volume  = {42},
  number  = {8},
  pages   = {2680--2706},
  year    = {2019},
  doi     = {10.1002/mma.5542}
}

@article{LiGuo2026,
author = {Yuequn Li and Fei Guo},
title = {Lifespan estimate for the semilinear regular Euler–Poisson–Darboux–Tricomi equation},
journal = {Applicable Analysis},
volume = {0},
number = {0},
pages = {1--22},
year = {2026},
publisher = {Taylor \& Francis},
doi = {10.1080/00036811.2026.2673374},
URL = {https://doi.org/10.1080/00036811.2026.2673374},
eprint = {https://doi.org/10.1080/00036811.2026.2673374}
}

@article{DAbbiccoPalmieri2021,
  author  = {D'Abbicco, Marcello and Palmieri, Alessandro},
  title   = {A Note on {$L^{p}$--$L^{q}$} Estimates for Semilinear
             Critical Dissipative {Klein--Gordon} Equations},
  journal = {Journal of Dynamics and Differential Equations},
  volume  = {33},
  number  = {1},
  pages   = {63--74},
  year    = {2021},
  doi     = {10.1007/s10884-019-09818-2}
}

@misc{ChenHamza2026,
  author        = {Chen, Wenhui and Hamza, Mohamed Ali},
  title         = {Blow-up criteria and lifespan estimates for semilinear
                   wave equations with time-dependent damping and mass},
  year          = {2026},
  eprint        = {2608.03713},
  archivePrefix = {arXiv},
  primaryClass  = {math.AP}
}

\end{document}